\documentclass[12pt]{amsart}

\usepackage[margin = 1in]{geometry}
\usepackage{amsfonts, amsmath,amscd, amssymb, cite, color, latexsym, mathrsfs, mathtools,
slashed, stmaryrd, verbatim, wasysym }
\usepackage[all]{xy}
\usepackage[pdftex]{graphicx}
\usepackage{hyperref}
\usepackage[applemac]{inputenc}

\newtheorem{theorem}{Theorem}

\newtheorem{lemma}[theorem]{Lemma}
\newtheorem{proposition}[theorem]{Proposition}
\newtheorem{remark}{Remark}

\numberwithin{equation}{section}
\numberwithin{theorem}{section}

\let\oldsqrt\sqrt
\def\sqrt{\mathpalette\DHLhksqrt}
\def\DHLhksqrt#1#2{%
\setbox0=\hbox{$#1\oldsqrt{#2\,}$}\dimen0=\ht0
\advance\dimen0-0.2\ht0
\setbox2=\hbox{\vrule height\ht0 depth -\dimen0}%
{\box0\lower0.4pt\box2}}

\newcommand{\df}[1]{\mathfrak{#1}}

\renewcommand{\bar}{\overline}
\renewcommand{\hat}[1]{\widehat{#1}}

\newcommand{\wt}[1]{\widetilde{#1}}
\newcommand{\rest}[1]{\big\rvert_{#1}} 

\newcommand\lra{\longrightarrow}

\newcommand\pa{\partial}

\newcommand\eps\varepsilon
\renewcommand\epsilon\varepsilon

\newcommand\CI{{\mathcal{C}}^{\infty}}

\newcommand\ang[1]{\langle #1 \rangle}
\newcommand\floor[1]{\lfloor #1 \rfloor}
\newcommand{\lrpar}[1]{\left( #1 \right)}
\newcommand{\lrspar}[1]{\left[ #1 \right]}

\newcommand{\norm}[1]{\lVert #1 \rVert}

\renewcommand\det{\operatorname{det}}

\newcommand\dvol{\operatorname{dvol}}

\DeclareMathOperator*{\FP}{\operatorname{FP}}

\renewcommand\Re{\operatorname{Re}}

\newcommand\Ric{\operatorname{Ric}}

\newcommand\scal{\mathrm{scal}}

\newcommand\Vol{\operatorname{Vol}}
\newcommand\Weyl{\operatorname{Weyl}}

\newcommand\Mand{\text{ and }}
\newcommand\Mas{\text{ as }}

\newcommand\Mev{\operatorname{even}}

\newcommand\Mforall{\text{ for all }}
\newcommand\Mforany{\text{ for any }}

\newcommand\Mif{\text{ if }}

\newcommand\Modd{\operatorname{odd}}

\newcommand\Motherwise{\text{ otherwise }}

\newcommand\Mwhere{\text{ where }}
\newcommand\Mwith{\text{ with }}

\newcommand\paperintro%
        {%
         }
\newcommand\paperbody%
        {%
         }

\newcommand\bA{\mathbf{A}}
\newcommand\bB{\mathbf{B}}

\newcommand\bbB{\mathbb{B}}

\newcommand\bbN{\mathbb{N}}

\newcommand\bbR{\mathbb{R}}
\newcommand\bbS{\mathbb{S}}

\newcommand\cB{\mathcal{B}}
\newcommand\cC{\mathcal{C}}
\newcommand\cD{\mathcal{D}}
\newcommand\cE{\mathcal{E}}

\newcommand\cG{\mathcal{G}}

\newcommand\cL{\mathcal{L}}
\newcommand\cM{\mathcal{M}}

\newcommand\cO{\mathcal{O}}

\newcommand\cR{\mathcal{R}}
\newcommand\cS{\mathcal{S}}

\newcommand\cU{\mathcal{U}}

\newcommand\sB{\mathscr{B}}
\newcommand\sC{\mathscr{C}}
\newcommand\sD{\mathscr{D}}
\newcommand\sE{\mathscr{E}}

\newcommand\sO{\mathscr{O}}

\newcommand\sQ{\mathscr{Q}}

\DeclareMathAlphabet{\mathpzc}{OT1}{pzc}{m}{it}

\title{Poincar\'e-Lovelock metrics on conformally compact manifolds}
\author{Pierre Albin}
\address{University of Illinois, Urbana-Champaign}
\email{palbin@illinois.edu}

\begin{document}

\begin{abstract}
An important tool in the study of conformal geometry, and the AdS/CFT correspondence in physics, is the Fefferman-Graham expansion of conformally compact Einstein metrics.
We show that conformally compact metrics satisfying a generalization of the Einstein equation, Poincar\'e-Lovelock metrics, also have Fefferman-Graham expansions. Moreover we show that conformal classes of metrics that are near that of the round metric on the $n$-sphere have fillings into the ball satisfying the Lovelock equation, extending the existence result of Graham-Lee for Einstein metrics.
\end{abstract}

\maketitle

\paperintro
\section*{Introduction}

The purpose of this paper is to show that an important part of the theory developed for Poincar\'e-Einstein metrics, metrics that are conformally compact and Einstein, holds also for Poincar\'e-Lovelock metrics, metrics that are conformally compact and Lovelock. Specifically we show that Poincar\'e-Lovelock metrics with sufficient boundary regularity on arbitrary manifolds have an asymptotic expansion identical in form to that of Poincar\'e-Einstein metrics and that conformal classes of metrics on the sphere sufficiently close to that of the round metric can be filled in with Poincar\'e-Lovelock metrics.\\

The local invariants of a Riemannian manifold are easy to write down. Weyl's invariant theory identifies them with the contractions of the Riemann curvature tensor and its covariant derivatives. On the other hand local scalar invariants of a conformal structure are less readily accessible. Inspired by the tight connection between the Riemannian geometry of hyperbolic space and the conformal geometry of the round sphere, the Fefferman-Graham \cite{Fefferman-Graham:Conf, Fefferman-Graham:Ambient} `ambient construction' seeks to invariantly associate to a manifold with a conformal structure another manifold with a Riemannian structure. Conformal invariants of the former are then obtained from Riemannian invariants of the latter,

A Riemannian manifold $(M,g)$ is conformally compact if $M$ is the interior of a manifold with boundary $\bar M$ and for some, hence any, non-negative function $x \in \CI(\bar M)$ that  vanishes simply and exactly at $\pa M,$ $x^2 g$ is a metric on $\bar M.$ The metric on $\pa M$ obtained by restricting $x^2g$ to $\pa M$ depends on the choice of $x,$ but different choices yield metrics in the same conformal class, the `conformal infinity'  of $g$ \cite[Chapter 9]{Penrose-Rindler}. 
The problem posed in \cite{Fefferman-Graham:Conf} is, given a conformal class of metrics on $\pa M,$ find a conformally compact Einstein metric $g$ whose conformal infinity is the given conformal class. These `Poincar\'e-Einstein metrics' can, for appropriate choices of $x,$ be written near the boundary as $x^{-2}(dx^2 + h)$ where $h$ has an asymptotic expansion of the form
\begin{equation}\label{eq:FGExp}
	h \sim
	\begin{cases}
	h_0 + h_2 x^2 + (\text{even powers}) + h_{n-1}x^{n-1} + h_{n}x^n + \ldots & \Mif n \Modd\\
	h_0 + h_2 x^2 + (\text{even powers}) + h_{n,1}x^{n} \log x + h_{n}x^n + \ldots & \Mif n \Mev
	\end{cases}
\end{equation}
with $n= \dim \pa M.$ (These `appropriate' $x$ are known as {\em special} boundary defining functions.)

The choice of $x$ determines a metric $h_0$ in the conformal infinity and Riemannian invariants that do not depend on such a choice are invariants of the conformal class of $h_0.$ An important example is the renormalized volume,
\begin{equation}\label{eq:RenVol}
	{}^{R}\Vol(M) = \FP_{s=0} \int_M x^s \; \dvol_g = \FP_{\eps=0} \int_{ \{x \geq \eps\} } \; \dvol_g,
\end{equation}
which for $n$ odd is independent of the choice of special boundary defining function used in its definition, while for $n$ even its dependence on $x$ is mediated through the term $h_{n,1}$ in the expansion of the metric.

The importance of the renormalized volume is that it plays a prominent r\^ole in the Anti-de-Sitter / Conformal Field Theory (briefly AdS/CFT) correspondence. This is a proposed duality \cite{Maldacena} between a quantum gravity theory in the interior of a manifold and a conformal field theory on the boundary. This duality was clarified in \cite{Gubser-Klebanov-Polyakov}, \cite{Witten:AdS} as an equivalence of partition functions and the renormalized volume shows up as the partition function of the gravity theory. The dependence on the choice of boundary defining function was shown to match the expected conformal anomaly of the conformal field theory on the boundary   when $n=2$ or $n=4$ \cite{Henningson-Skenderis}.\\

A natural generalization arises from recalling that in four dimensions the only natural tensors on Riemannian manifolds that are symmetric, built up from the metric and its first two derivatives, and divergence-free are linear combinations of the metric and its Einstein tensor,
\begin{equation*}
	a g_{ij} + b E_{ij}(g), \quad
	E_{ij}(g) = \Ric(g)_{ij} -\frac{\scal(g)}2 g_{ij}.
\end{equation*}
Indeed, this is one of the motivations for the form of the field equations of gravity in general relativity.
It was shown by Lovelock \cite{Lovelock} that in dimension $m,$ the space of tensors satisfying these properties has dimension $\floor{\frac m2}$ (though only the metric and the Einstein tensor are {\em linear} in the second derivatives of the metric).
Generators for the other tensors are given by
\begin{multline*}
	E^{(2q)}_{ij}(g) = \Ric^{(2q)}_{ij} - \frac{ \scal^{(2q)}(g) }{2q} g_{ij} 
	\Mwhere
	\Ric^{(2q)}_{ij} = \delta^{\alpha_1 \alpha_2 \cdots \alpha_{2q}}_{i \beta_2 \cdots \beta_{2q}} 
	R_{\alpha_1\alpha_2j}^{\beta_2}
	R_{\alpha_3\alpha_4}^{\beta_3\beta_4} \cdots
	R_{\alpha_{2q-1}\alpha_{2q}}^{\beta_{2q-1}\beta_{2q}}, \\
	\scal^{(2q)}(g) = g^{st}\Ric^{(2q)}_{st}, \Mand
	\delta^{\alpha_1\cdots \alpha_{2q}}_{\beta_1\cdots \beta_{2q}} = \det( (\delta^{\alpha_i}_{\beta_j}) ).
\end{multline*}

\begin{remark}
For locally conformally flat metrics, we have
\begin{equation*}
	\mathrm{scal}^{(2q)}(g) = 
	\sigma_{2q}(g^{-1}P(g)),
\end{equation*}
the $(2q)^{\text{th}}$ elementary symmetric function of the eigenvalues of the Schouten tensor of $g,$ see \eqref{eq:SigmaKh} and Remark \ref{rmk:Effectives}.
\end{remark}

Divergence-free symmetric two tensors natural in the metric and its first two derivatives are known as generalized Einstein tensors, or Lovelock tensors. 
We will refer to a metric that is conformally compact and satisfies an equation of the form
\begin{equation}\label{eq:IntroLovelock}
	\sum \alpha_q E^{(2q)}_{ij}(g) = \lambda g_{ij},
\end{equation}
as a Poincar\'e-Lovelock metric.
For our purposes the particular values of the coefficients will be immaterial as long as they satisfy a single linear restriction.

For a fixed $n\geq 3,$ and any choice of scalars $\alpha = (\alpha_1, \ldots, \alpha_{\floor{\tfrac {n+1}2}})$ let
\begin{equation*}
	\lambda(\alpha) 
	= \sum \alpha_q \lrpar{-\frac12}^q \frac{n!(2q)!}{(n-2q+1)!} 
	= \sum \alpha_q \lambda^{(2q)},
\end{equation*}
chosen so that \eqref{eq:IntroLovelock} holds with $\lambda=\lambda(\alpha)$ for a hyperbolic metric.
Let $\mathrm{LimSec}(\alpha)$ be the set of $\kappa > 0$ such that
\begin{equation*}
\begin{gathered}
	\sum \alpha_q \lrpar{-\frac\kappa2}^q \frac{n!(2q)!}{(n-2q+1)!}
	= \lambda(\alpha),\\
	\bA_1(\alpha, \kappa) 
	= \sum \alpha_q \lrpar{-\frac\kappa2}^{q-1} \frac{(n-2)!}{2}\frac{(2q)!}{(n-2q)!} \neq 0.
\end{gathered}
\end{equation*}
The usual Einstein equation corresponds to $\alpha = (1,0,\ldots, 0),$ $\lambda(\alpha) = -n,$ $\mathrm{LimSec}(\alpha) = \{1\}.$
In general the number of elements in $\mathrm{LimSec}(\alpha)$ can be any number in $\{0,\ldots, \floor{\tfrac {n+1}2}\},$ but if the signs of the $\alpha_i$ alternate then $\mathrm{LimSec}(\alpha) = \{1\}.$ 
We will assume that $\mathrm{LimSec}(\alpha)\neq\emptyset.$

\begin{theorem}\label{thm:FefGrExp}
Let $X$ be an $n$-dimensional closed manifold, $n\geq 3$ with a conformal class of Riemannian metrics $\df c,$ and fix $\alpha$ such that $\mathrm{LimSec}(\alpha)\neq \emptyset.$\\

a)
Choose a locally constant function $\bar\kappa: X \lra \mathrm{LimSec}(\alpha).$
Let $N = n-2$ if $n$ is even and $N= \infty$ if $n$ is odd.
There is a conformally compact Riemannian metric $g$ on $X \times [0,1]_x$ with conformal boundary $X \times \{0\},$ whose sectional curvatures converge to $-\bar\kappa$ as $x\to 0,$ which is even modulo $\cO(x^{N+2})$ and asymptotically satisfies the Lovelock equation
\begin{equation*}
	\sum \alpha_q E^{(2q)}_{ij}(g) = \lambda(\alpha) g_{ij} + \cO(x^{N}).
\end{equation*}
Moreover, $g$ is unique modulo $\cO(x^{N})$ up to a diffeomorphism fixing $X \times \{0\}.$

For any Riemannian metric $h_0$ in the conformal class $\df c$ there is a boundary defining function $x$ for which $g$ takes the form $x^{-2}(\bar\kappa^{-1}dx^2+h(x))$ with $h(0) =h_0$ and the tensors $\{\pa_x^i h(0)\}_{i \in \{0,\ldots, N+1\}}$ are formally determined by $h_0.$\\

b)
Assume that $(M,g)$ is a conformally compact manifold with conformal boundary $(X, \df c),$ and $g$ satisfies the Lovelock equation
\begin{equation*}
	\sum \alpha_q E^{(2q)}_{ij}(g) = \lambda(\alpha) g_{ij}.
\end{equation*}
Then the sectional curvatures of $M$ converge to $-\bar\kappa,$ with $\bar\kappa:X \lra \mathrm{LimSec}(\alpha)$ a locally constant function, and we can find $x$ such that $g$ has the form $x^{-2}(\bar\kappa^{-1}dx^2+h_x)$ near the boundary.
Moreover, if $g$ has sufficient boundary regularity,\\
i) $h$ has an expansion of the form \eqref{eq:FGExp} where the tensors $h_k$ for $k<n,$ and $h_{n,1}$ if $n$ is even, are formally determined by $h_0.$ \\
ii) The tensor $h_n$ is not formally determined by $h_0;$ if $n$ is odd $h_n$ is trace free, if $n$ is even its trace is formally determined by $h_0.$\\
iii) if $n$ is odd $h_n$ is divergence free, if $n$ is even its divergence is formally determined by $h_0.$\\
In any case all of the tensors in the expansion are formally determined by  $h_0$ and $h_n.$

\end{theorem}

\begin{remark}
If $\alpha = (1,0,\ldots,0)$ then the Lovelock equation is the Einstein equation and this theorem is the usual Fefferman-Graham expansion. In this case the boundary regularity of $g$ in (b) is shown in \cite{Chrusciel-Delay-Lee-Skinner}.
\end{remark}

It turns out that the formal determination of the asymptotic expansion of a conformally compact metric holds for a larger family of curvature equations, obtained by modifying the trace of the Lovelock tensors,
\begin{equation*}
	F_g(\alpha,\beta) 
	= \sum \alpha_q (\Ric^{(2q)}(g)  - \lambda^{(2q)}g) + \beta_q (\scal^{(2q)}(g)-(n+1)\lambda^{(2q)})g = 0,
\end{equation*}
which reduces to the Lovelock equation above if $\beta_q = -\frac{\alpha_q}{2q}$ for all $q.$

\begin{theorem}\label{thm:FefGrExp2}
Parts (b)(i) and (b)(ii) of Theorem \ref{thm:FefGrExp} hold for metrics satisfying $F_g(\alpha, \beta)=0$ as long as $\alpha \neq -(n+1)\beta$ and $\mathrm{LimSec}(\alpha, \beta) \neq\emptyset.$
\end{theorem}

If $(n+1)>4$ and $g$ is a solution of $F_g(\alpha, \beta)=0,$ such as a Poincar\'e-Lovelock metric, then 
\begin{equation*}
	g = x^{-2}(\kappa^{-1}dx^2+h_0 + h_2x^2 + h_4 x^4 + \cO(x^5)).
\end{equation*}
We determine the tensors $h_2$ and $h_4$ below in \S\ref{sec:FirstCouple}.
The tensor $h_2$ is always a multiple of the Schouten tensor of $h_0,$
\begin{equation*}
	h_2 = -\frac1{\kappa} P(h_0),
\end{equation*}
while the tensor $h_4$ is more complicated, 
\begin{multline*}
	h_4 
	= 
	-\frac1{(n-4)}
	\Big( 
	-\frac{h_0\sC_h(\dot{\Ric})}{4\kappa(n-1)}
	+h_0(\tfrac14 \sC_{h_0}^2(h_2^2)
	- \tfrac12 \sC_{h_0}(h_2)^2)
	- \sC_{h_0}(h_2^2) 
	+2 h_2\sC_{h_0}(h_2) 
	+\frac{\dot{\Ric}}{\kappa}
	\Big) \\
	-\frac{\bB_{1,2}(\alpha, \beta, \kappa)h_0}{(n-4) \bA_1(\alpha, \kappa)}	
	\Big(
	\tfrac12 (2n-5) \sC_{h_0}^2(h_2^2)
	- 2(n-2)\sC_{h_0}(h_2)^2 
	+ \frac{\dot{\scal}}{\kappa} \Big)\\
	+2(n-3)
	\Big(
	-\frac14\sC_{h_0}^2(h_2^2) + \frac12\sC_{h_0}(h_2)^2 
	-\frac1{4\kappa(n-1)}\sC_h(\dot{\Ric})
	-\frac{\bA_3(\alpha, \kappa)}{4\kappa(n-1)\bA_1(\alpha, \kappa)}\sC_{h_0}^4(\Weyl_{h_0}^2)
	 \Big) \\
	 \\
	-\frac{
	4(n-1)\bA_3(\alpha, \kappa) \sC_{h_0}^3( \Weyl_{h_0}^2) 
	+ (4(n-1) \bB_{3,4}(\alpha, \beta, \kappa)-\bA_3(\alpha, \kappa))
		h_0\sC_{h_0}^4( \Weyl_{h_0}^2 ) 
	}{4\kappa(n-1)(n-4)\bA_1(\alpha, \kappa)} 
\end{multline*} 
where
\begin{equation*}
\begin{gathered}
	\dot{\Ric} 
	= \frac12\Delta_{L,h_0}(h_2) - \delta_{h_0}^*(\delta_{h_0}h_2) - \frac12\mathrm{Hess}_{h_0}(\sC_{h_0}(h_2)),\\
	\dot{\scal} 
	= \sC_{h_0}(\dot{\Ric})
	+\tfrac12\sC_{h_0}^2(\Ric \owedge h_2)-\sC_{h_0}(h_2)\scal,
\end{gathered}
\end{equation*}
and we are using the double form formalism reviewed in \S\ref{sec:DoubleForms}, and functions of $\alpha,$ $\beta$ specified in \S\ref{sec:FG}.\\

An advantage of the Poincar\'e-Lovelock metrics over other solutions of  $F_g(\alpha, \beta)=0$ is that the former are guaranteed to exist, at least on the ball, by the following analogue of \cite[Theorem A]{Graham-Lee}.

\begin{theorem}
Let $M = \bbB^{n+1},$ $n\geq 4,$ ${\df h}$ the hyperbolic metric on $M$ and $\hat {\df h} = \rho^2{\df h}|_{\{\rho=0\}}$ the round metric on $\bbS^n = \pa M.$ Let $\alpha$ be such that $\mathrm{LimSec}(\alpha)\neq\emptyset.$

For any smooth Riemannian metric $\hat g$ on $\bbS^{n}$ which, for some $\theta>0,$ is sufficiently close in $\cC^{2,\theta}(M, \cS^2(M))$ to $\hat {\df h}$ there is a metric $g \in \cC^{\infty}(M, \cS^2(M))\cap \rho^{-2}\cC^{0}(\bar M, \cS^2(M))$ satisfying 
\begin{equation*}
	\begin{cases}
	\sum \alpha_q E^{(2q)}_{ij}(g) = \lambda(\alpha) g_{ij}, \\
	x^2 g\rest{\pa M} \text{ is conformal to } \hat g.
	\end{cases}
\end{equation*}
\end{theorem}

The Lovelock equations are generally not elliptic, even after gauge-fixing, and hence can behave very differently to the Einstein equations. For example, the product of an $n$-dimensional Riemannian manifold and the $\ell$-dimensional flat torus satisfies $E^{(2q)}(g)=0$ whenever $2q>n,$ so that for many Lovelock equations the moduli space of solutions is infinite dimensional.
However it turns out that the linearization of the Lovelock equations at the hyperbolic metric on the ball is, as long as $\bA_1(\alpha, \kappa) \neq 0,$ essentially the same as the linearization of the Einstein equations. 

\begin{remark}
We do not explore the consequences of the Lovelock equations with $\bA_1(\alpha,\kappa)=0.$
The Lovelock equation in this case shows that the trace of $\pa_x h|_{x=0}$ vanishes but does not determine its trace-free part, while in the Graham-Lee argument for existence the vanishing of $\bA_1(\alpha,\kappa)$ implies the vanishing of the linearization of the Lovelock equations at a hyperbolic metric.
\end{remark}

There are many papers in the literature that discuss  modifications of the Einstein equation.
Recently, for example,  Alaee and Woolgar \cite{Alaee-Woolgar} consider asymptotically hyperbolic metrics satisfying the Bach equation in dimension four and a modification in higher dimensions and derive their formal power series expansions, while in \cite{Chernicoff:Q} the authors consider higher curvature theories of gravity whose actions are given by generalizations of Branson's Q-curvature.

In the context of the AdS/CFT correspondence, there is a systematic discussion of asymptotic expansions of solutions of higher derivative theories in three dimensional gravity in \cite{Skenderis-Taylor-vanRees}.
Four-dimensional theories are treated in, e.g., \cite{Smolic2013}.
The paper
\cite{ImbimboSchwimmerTheisenYankielowicz}
(cf. \cite{Skenderis:AAdS})
discusses how the coefficients of the expansion of a conformally compact metric are constrained by their behavior under conformal transformations regardless of the gravitational equation imposed (assuming that the expansion is smooth and that the gravitational expansion is satisfied by hyperbolic space). 
Note that Fefferman-Graham \cite[Proposition 3.5]{Fefferman-Graham:Ambient} show that for the Einstein equation only contractions of the Ricci curvature and its covariant derivatives show up, while, e.g.,  the expression for $h_4$ above shows that the Weyl curvature is involved in the expansion of solutions of general Lovelock equations.

We mention a few papers that are more specifically in the setting of Lovelock gravity in the AdS/CFT correspondence.
In \cite{Kofinas-Olea}
boundary terms consistent with the Lovelock action and AdS asymptotics are determined.
In \cite{deBoer2010}
the authors consider $AdS_7/CFT_6$ and explain how considerations in a conformal field theory hypothetically dual to a Lovelock theory restrict the physically meaningful values of the coupling constant vector $\alpha.$ This theme is also explored in \cite{Camanho2010} 
for cubic Lovelock gravity in arbitrary dimensions.
In \cite{Camanho2013} 
the authors point out that the inclusion of `higher curvature terms' allows for the description of more general conformal field theories.
In \cite{Aksteiner2016} the authors consider actions that are up to quadratic in the curvature and they identify specific values of the couplings for which the Lovelock equations do not determine the terms in the expansion of the metric; this seems to correspond to the condition $\bA_1(\alpha, \kappa)=0$ above. In {\em loc. cit.} the authors point out that in five dimensions this corresponds to `gravitational Chern-Simons theory'.

$ $\\
{\bf Consequences}\\
We briefly review some of the immediate consequences of Theorem \ref{thm:FefGrExp}; for a more complete survey of these consequences in the Einstein setting see, e.g., \cite{Djadli-Guillarmou-Herzlich}.

As mentioned above, if $(M,g)$ is a Poincar\'e-Einstein manifold then an important conformal invariant of its boundary is the renormalized volume \eqref{eq:RenVol}. In \cite{Albin:RenInt} it is shown that every scalar Riemannian invariant of $(M,g)$ has a renormalized integral that is independent of the choice of special boundary defining function used in its definition. As this only depends on the form of the Fefferman-Graham expansion it holds for all Poincar\'e-Lovelock metrics.

A particularly interesting example is the Pfaffian, the integrand of the Gauss-Bonnet theorem, for which we have \cite[Theorem 1.2]{Albin:RenInt}
\begin{equation*}
	{}^R\int_M \mathrm{Pff}\; \dvol_g = \chi(M).
\end{equation*}
It is natural to wonder if this is an index theorem but the relevant elliptic operator, the Gauss-Bonnet operator, $\eth_{GB},$ is shown {\em not} to be Fredholm on any conformally compact manifold in \cite{Mazzeo:Hodge}. 
Nevertheless a renormalized index is defined in \cite{Albin:RenInd} using renormalized integrals and shown to satisfy
\begin{equation*}
	{}^R\mathrm{ind}(\eth_{GB}) = {}^R\int_M \mathrm{Pff}\; \dvol_g.
\end{equation*}
Indeed a renormalized index theorem is proven for all Dirac-type operators on conformally compact manifolds.
The renormalized supertrace of the heat kernel is only guaranteed to be independent of the choice of special boundary defining function if the metric is even to order $n+1,$ so to one order greater than the general Poincar\'e-Lovelock metric.
(Most Dirac-type operators on conformally compact manifolds can not even be smoothly perturbed to be Fredholm \cite{Albin-Melrose:Fred1}.
An index formula for elliptic pseudodifferential operators on conformally compact manifolds that {\em are} Fredholm is established in \cite{Albin-Melrose:RelChern}.)

For any conformally compact metric $g,$ whose sectional curvatures converge to a locally constant function at $\pa M,$ the resolvent
\begin{equation*}
	R(s) = (\Delta - s(n-s))^{-1}
\end{equation*}
is constructed by Mazzeo and Melrose \cite{Mazzeo-Melrose:Zero} as an analytic family of bounded operators on $L^2$ for $\Re(s)>n.$ 
In {\em loc. cit.} they show that its Schwartz kernel extends as a meromorphic function to the complex plane minus a discrete set.
Guillop\'e and Zworski \cite{Guillope-Zworksi} showed that for a conformally compact metric with constant curvature near infinity the extension is to the whole complex plane. The general case was understood by Guillarmou \cite{Guillarmou:Mero} who showed that if the metric is even modulo $\cO(x^{2k+1})$ then the resolvent 
extends meromorphically to $\Re(s)>(n-2k-1)/2.$ (A different approach has subsequently been developed by Vasy \cite{Vasy:AH}.)
Thus for Poincar\'e-Einstein and Poincar\'e-Lovelock metrics the resolvent is a meromorphic function for $\Re(s)>0.$

Using the resolvent it can be shown that, given a function $f \in \CI(\pa M),$ and $s$ such that 
\begin{equation*}
	\Re(s)\geq n/2, \quad
	2s-n \notin \bbN_0,  \Mand 
	\text{$s(n-s)$ is not a pole of $R(s),$}
\end{equation*}
there is a unique solution of the equation $(\Delta-s(n-s))u=0$ of the form
\begin{equation*}
	u = x^{n-s} F(x,y) + x^{s}G(x,y)
\end{equation*}
with $F, G \in \CI(M)$ and $F(0,y) = f.$
The scattering matrix at energy $s,$ $S(s),$ is the map that sends $f$ to $G(0,y)$ \cite{Joshi-SaBarreto} and makes up a meromorphic family of pseudodifferential operators on $\pa M$ (cf. \cite[\S 5]{deHaroSkenderisSolodukhin}).
Graham and Zworski \cite{Graham-Zworski} show that an appropriate multiple of the residue of $S(s)$ at $s=n/2+k,$
\begin{equation*}
	P_k = (-1)^{k+1}(2^{2k}k!(k-1)!) \mathrm{Res}_{s=n/2+k} S(s)
\end{equation*}
(with $k \in \bbN,$ $k\leq n/2$ if $n$ is even, and under a generic assumption on $g$)
are conformally covariant self-adjoint differential operators on $\pa M$ whose principal part is the same as the $k^{\text{th}}$ power of the Laplacian $\Delta^k.$ 
They show \cite[\S 4]{Graham-Zworski} that these operators can also be obtained by formal power series arguments and coincide with the GJMS operators \cite{GJMS}.

Assuming now that $n$ is even,  it follows from the asymptotic expansion of the Laplacian that $P_{n/2}1=0$ so that $S(s)1$ does not have a pole at $s=n.$
The scalar Riemannian invariant
\begin{equation*}
	Q = (-1)^{n/2}(2^{n}(\tfrac n2)!(\tfrac n2-1)!) S(n)1
\end{equation*}
is known as Branson's Q-curvature. If we denote the Q-curvatures of $h$ and $\hat h = e^{2\Upsilon}h$ by $Q$ and $\hat Q$ respectively, these are related by
\begin{equation*}
	e^{n\Upsilon}\hat Q = Q + P_{n/2}\Upsilon.
\end{equation*}
The integral of Q-curvature is (thus) conformally invariant and Graham-Zworski show that if one writes
\begin{equation*}
	\Vol_g(\{x\geq \eps\}) = c_0 \eps^{-n} + c_2\eps^{-n+2} + \ldots + c_{n-2}\eps^{-2} + L \log (\tfrac1\eps) + {}^R\Vol(M) + o(1)
\end{equation*}
then $L$ is the integral of $2S(n)1,$ hence a multiple of the integral of Q-curvature.

In \cite{Fefferman-Graham:Q}, Fefferman and Graham make use of the work of \cite{Graham-Zworski} and define a Q-curvature in odd dimensions whose integral is a multiple of the renormalized volume. (In \cite{Chang-Qing-Yang} this is related to the Gauss-Bonnet theorem.)

The theorems in \cite{Graham-Zworski, Fefferman-Graham:Q} only make use of the Einstein equation through the form of the expansion of the metric \eqref{eq:FGExp} and so hold also for Poincar\'e-Lovelock metrics. Thus for each choice of $\alpha$ such that $\mathrm{LimSec}(\alpha) \neq \emptyset,$ there are GJMS operators with the same leading part and conformal covariance and there is a Q-curvature with the corresponding conformal transformation law whose integral appears in the asymptotic expansion of the volume.

$ $\\
The contents of the paper are as follows. In section \ref{sec:LovelockTensors} we discuss Lovelock tensors using the formalism of double forms. This was introduced by Kulkarni \cite{Kulkarni} and has recently been developed by Labbi \cite{Labbi:pq} -- \cite{Labbi:Weit}. In section \ref{sec:FG} we apply this formalism to find the formal asymptotic expansion of solutions to the equation $F_g(\alpha, \beta)=0$ mentioned above. This is analogous to the treatment of the Einstein equation in, e.g., \cite{Graham:Vol, Graham-Hirachi:Obst}. We then parallel \cite[\S6.9]{Juhl:Book} in \S\ref{sec:FirstCouple} to compute the first couple of non-zero tensors in the expansion of a Poincar\'e-Lovelock metric. In section \ref{sec:GL} we turn to the existence result. We follow \cite{DeLimaSantos:Defs} to compute the linearization of the gauge-fixed Lovelock equation and then use the results of \cite{Graham-Lee}.\\

{\bf Acknowledgements.}
This work was supported by NSF grant DMS-1711325. I am happy to acknowledge useful conversations and encouragement from Rafe Mazzeo, Robin Graham, Guofang Wang,  and especially Marika Taylor and Kostas Skenderis to whom I am indebted for pointers to the physics literature.
I am also grateful to the Banff International Research Station and the organizers of the workshop `Asymptotically Hyperbolic Manifolds' held in May 2018.  \\

\tableofcontents

\paperbody
\section{Lovelock tensors and double forms} \label{sec:LovelockTensors}

\subsection{Lovelock tensors} 

Certain problems in statistics (fitting regression equations non-linear in parameters) led Hotelling to pose the problem of determining the volume of a small tube around a manifold $\wt M$ embedded in $\bbR^N,$
\begin{equation*}
	\cB_{\eps}(\wt M) = \{ r \in \bbR^N: \text{ distance from $r$ to $\wt M$ is less than }\eps \}.
\end{equation*}
In 1937 Weyl attended a seminar where Hotelling gave a solution for submanifolds $\wt M$ of codimension one \cite{Hotelling} and the following year Weyl gave a solution for arbitrary codimension \cite{Weyl:Tubes}.
He showed that, for small $\eps,$ the volume of $\cB_{\eps}(\wt M)$ is a polynomial 
\begin{equation*}
	\Vol(\cB_{\eps}(\wt M)) = \Vol(\bbB_{\eps}^{N-m})
	\sum_{q=0}^{\floor{\tfrac m2}} \frac{\eps^{2q}}{(N- m + 2)(N- m + 4) \cdots (N- m + 2q)} 
	\lrspar{ \int_{\wt M} \wt\ell_{2q}(g) \; \mathrm{dVol}_{g} },
\end{equation*}
where $m = \dim \wt M,$ $\bbB_{\eps}^{N-m}$ denotes a ball of radius $\eps$ in $\bbR^{N-m},$ and
the coefficients are integral invariants of $\wt M$ with its induced Riemannian metric $g$---hence are independent of the particular embedding.
The integrands, $\wt\ell_{2q}(g),$ are known by many names, e.g., `Weyl volume-of-tube invariants', `Lipshitz-Killing curvatures', 
and `Lovelock scalars', the latter because they essentially coincide with the traces of the Lovelock tensors mentioned in the introduction,
\begin{equation*}
	\wt\ell_{2q}(g) = \frac{\scal^{(2q)}(g)}{(2q)!q!}.
\end{equation*}
The first few are given by
\begin{equation*}
	\wt\ell_0(g) = 1, \quad
	\wt\ell_2(g) = \frac12\scal, \quad
	\wt\ell_4(g) = \frac18(|R|^2-4|\Ric|^2 + \scal^2).
\end{equation*}
Another name for these invariants is `Gauss-Bonnet curvatures' as $\wt\ell_{2q}(g)$ is, after multiplying by $(2\pi)^q,$ the integrand of the Gauss-Bonnet theorem in dimension $2q,$ i.e., the $2q$-dimensional Pfaffian.
This observation was used by Allendoerfer and Weil in the original proof of the Gauss-Bonnet theorem \cite{Allendoerfer, Allendoerfer-Weil}. 

These invariants have connections to many topics in geometry and physics. They appear, for example, in Chern's kinematic formul\ae{} for quermassintegrals \cite{Chern:Kinematic}, Steiner's formula \cite[Chapter 10]{Gray:Tubes}, and an approach to lattice gravity \cite{Cheeger-Muller-Schrader:Lattice, Cheeger-Muller-Schrader:Curvature, Cheeger-Muller-Schrader:Kinematic}. 
For a modern discussion see the book \cite{Gray:Tubes}.\\

Just as each  $\scal^{(2q)}(g)$ is a generalization of the scalar curvature, the functional
\begin{equation*}
	g \mapsto  \int_{\wt M} \scal^{(2q)}(g) \; \mathrm{dVol}_{g}
\end{equation*}
generalizes the Einstein-Hilbert action and  its Euler-Lagrange derivative (after multiplying by $-q^{-1}$), $E^{(2q)}(g),$ known as the $(2q)-$Lovelock tensor, generalizes the Einstein tensor.
On a manifold of dimension $m$ the functions $\scal^{(2q)}(g)$ vanish identically if $2q>m$ (see \eqref{eq:DimVanish} below), while if $m$ is even the scalar $\scal^{(m)}(g)$ is essentially the Pfaffian of the curvature of $g$ and hence its Euler-Lagrange derivative is identically zero.

Directly from their definition, the Lovelock tensors are symmetric divergence-free $(0,2)$-tensors (e.g., \cite[Proposition 4.11]{Besse}) that only depend on the metric and its first two covariant derivatives (i.e., its curvature).
Lovelock \cite{Lovelock} showed that every $(0,2)$-tensor satisfying these properties is in the $\bbR$-span of $\{E^0(g), \ldots, E^{\floor{m/2}}(g)\},$ which is now known as the space of Lovelock tensors.

Lovelock tensors satisfy Schur's Lemma: if for some metric $g$ some non-zero $\bbR$-linear combination of the Lovelock tensors is equal to the product of a scalar function with the metric,
\begin{equation*}
	\sum \alpha_q E^{(2q)}_{ij}(g) = f g_{ij},
\end{equation*}
then that scalar function $f$ must be locally constant. We refer to such metrics as {\bf Lovelock metrics}.

\subsection{Double forms} \label{sec:DoubleForms}

The formalism of double forms studied by Kulkarni \cite{Kulkarni} is very convenient for analyzing Lovelock tensors and scalars.
It has recently been developed in various articles of Labbi \cite{Labbi:pq, Labbi:OnGB, Labbi:Var, Labbi:OnGr, Labbi:Rmks, Labbi:Weit}.

On a Riemannian manifold $(M,g)$ of dimension $m,$ an $(a,b)$-form is an element of
\begin{equation*}
	\Omega^{a \otimes b}(M) =
	\CI(M; \Lambda^a T^*M \otimes \Lambda^b T^*M),
\end{equation*}
and a double form is an element of the direct sum of the $(a,b)$-forms,
\begin{equation*}
	\Omega^{*\otimes *}(M) = \bigoplus_{a,b} \Omega^{a\otimes b}(M).
\end{equation*}
The wedge product induces a product on double forms by extending 
\begin{equation*}
	(\alpha \otimes \beta)(\gamma \otimes \delta) = (\alpha \wedge \gamma) \otimes (\beta \wedge \delta)
\end{equation*}
from simple tensors to all of $\Omega^{*\otimes*}(M)$ by linearity.
This is known as the Kulkarni-Nomizu product, is often denoted $\owedge,$ and satisfies
\begin{equation*}
	\omega \in \Omega^{p\otimes q}(M), \;
	\theta \in \Omega^{r\otimes s}(M) \implies
	\omega\theta = (-1)^{pr+qs}\theta\omega.
\end{equation*}
In particular multiplication in $\bigoplus_a \Omega^{a\otimes a}(M)$ is commutative.

An important operation on double forms is contraction
\begin{equation*}
	\sC_g: \Omega^{r\otimes s}(M) \lra \Omega^{(r-1) \otimes (s-1)}(M).
\end{equation*}
If $r=0$ or $s=0,$ we set $\sC_g\omega=0$ for every $\omega \in \Omega^{r\otimes s}(M).$
Otherwise, for any vector fields $V_1, \ldots, V_{r-1}$ and $W_1, \ldots, W_{s-1},$ we set
\begin{equation*}
	\sC_g\omega((V_1, \ldots, V_{r-1}), (W_1, \ldots, W_{s-1}))
	= \sum \omega((e_j, V_1, \ldots, V_{r-1}), (e_j,W_1, \ldots, W_{s-1}))
\end{equation*}
where the sum runs over a $g$-orthonormal basis of vector fields, $\{e_j\}.$

For example, if $\omega, \theta \in \Omega^{1\otimes 1}(M)$ are given in a local coordinates by
\begin{equation*}
	\omega = \omega_{a,b} \; \theta^a \otimes \theta^b, \quad
	\eta = \eta_{a,b} \; \theta^a \otimes \theta^b,
\end{equation*}
then we have
\begin{equation}\label{eq:Example}
\begin{gathered}
	(\omega\eta)_{s\sigma, t\tau}
	=(\omega \owedge \eta)_{s\sigma, t\tau}
	=  \omega_{s,t}\eta_{\sigma,\tau} - \omega_{s,\tau}\eta_{\sigma, t} 
	- \omega_{\sigma, t} \eta_{s, \tau} + \omega_{\sigma, \tau} \eta_{s,t}, \\
\begin{multlined}
	\sC_g(\omega  \eta)_{i,j}
	= g^{ab}(\omega_{a,b}\eta_{i,j} - \omega_{a,j}\eta_{i,b} - \omega_{i,b}\eta_{a,j} + \omega_{s,t}\eta_{a,b}) 
		\phantom{Filler space formatting}\\
	= \sC_g(\omega) \eta_{i,j} + \sC_g(\eta) \omega_{i,j} - g^{ab}(\omega_{a,j}\eta_{i,b}  + \omega_{i,b}\eta_{a,j}), 
\end{multlined} \\
	\sC_g^2(\omega  \eta)
	= 2(\sC_g(\omega)\sC_g(\eta) - g^{ab}g^{ij}\omega_{a,j}\eta_{i,b}).
\end{gathered}
\end{equation}
Further, by considering an eigenbasis of the operator induced by $\omega,$ it is easy to see that
the complete contraction of $\omega^k$ is equal to the $k^{\text{th}}$ elementary symmetric polynomial of its eigenvalues,
\begin{equation}\label{eq:SigmaKh}
	\sC_g^k(\omega^k) = \sigma_k(g^{-1}\omega).
\end{equation}
$ $\\

The metric $g$ is naturally seen as a $(1,1)$-form, which we continue to denote $g,$
\begin{equation*}
	g(V)(W) = g(V,W).
\end{equation*}
The curvature of $g,$ $\mathrm{R},$ defines a $(2,2)$-form by
\begin{equation*}
	R_g \in \Omega^{2\otimes2}(M), \quad
	R_g((V_1, V_2), (W_1, W_2)) = g(\mathrm R(V_1, V_2)W_1, W_2).
\end{equation*}

The computation of the Weyl volume of tube invariants in \cite[Chapter 4]{Gray:Tubes} shows that
\begin{equation*}
	\scal^{(2q)}(g) = \sC_g^{(2q)}(R_g^q).
\end{equation*}
The tensor $\Ric^{(2q)}$ from the introduction corresponds to the $(1,1)$-form,
\begin{equation*}
	\cR^{(2q)}_g = \sC_g^{2q-1}R^q_g,
\end{equation*}
and the $(2q)$-Lovelock tensor, $E^{(2q)}(g),$ corresponds to the $(1,1)$-form
\begin{equation*}
	\cE^{(2q)}_g = \cR^{(2q)}_g - \frac{\scal^{(2q)}(g)}{2q} g.
\end{equation*}
As mentioned above, Lovelock \cite{Lovelock} showed (see also \cite{Labbi:Var}) that ($-q^{-1}$-times) the Euler-Lagrange derivative of $\int  \scal^{(2q)}(g) \; \dvol_g$ is $\cE^{(2q)}_g.$

Note that $\Lambda^pT^*M=0$ for $p>m$ implies that
\begin{equation}\label{eq:DimVanish}
	R_g^{\ell}=0, \quad \cR_g^{(2\ell)}=0, \quad \scal^{(2\ell)}(g)=0, 
	\text{ whenever } 2\ell>m.
\end{equation}
A useful observation is that that curvature $(2,2)$-form of a metric $g$ whose sectional curvature is identically equal to a constant $\kappa$ is given by
\begin{equation}\label{eq:CurvCstSec}
	R_g = \frac\kappa2 g^2.
\end{equation}

\begin{remark} \cite{Kulkarni}
A double form $\omega \in \Omega^{a\otimes b}(M)$ is {\em symmetric} if $a=b$ and 
\begin{equation*}
	\omega((V_1, \ldots, V_a), (W_1, \ldots, W_a)) = 
	\omega((W_1, \ldots, W_a), (V_1, \ldots, V_a))
\end{equation*}
for any vector fields. Symmetry is preserved by multiplication and contraction.

A double form satisfies the {\em first Bianchi identity} if it is in the null space of the operator
\begin{equation*}
\begin{gathered}
	\sB_1: \Omega^{a\otimes b}(M) \lra \Omega^{(a+1)\otimes (b-1)}(M), \\
\begin{multlined}
	\sB_1\omega((V_1, \ldots, V_{a+1}), (W_1, \ldots, W_{b-1})) \\
	= \sum (-1)^j \omega((V_1, \ldots, \hat V_j, \ldots, V_{a+1}), (V_j, W_1, \ldots, W_{b-1}))
\end{multlined}
\end{gathered}
\end{equation*}
and the {\em second Bianchi identity} if it is in the null space of the operator
\begin{equation*}
\begin{gathered}
	\sB_2: \Omega^{a\otimes b}(M) \lra \Omega^{(a+1)\otimes b}(M), \\
\begin{multlined}
	\sB_2\omega((V_1, \ldots, V_{a+1}), (W_1, \ldots, W_{b})) \\
	= \sum (-1)^j (\nabla_{V_j}\omega)((V_1, \ldots, \hat V_j, \ldots, V_{a+1}), (V_j, W_1, \ldots, W_{b-1})).
\end{multlined}
\end{gathered}
\end{equation*}
The null spaces of these operators are preserved by multiplication and that of $\sB_1$ is preserved by contraction. 

The metric and the curvature are symmetric $(1,1)$ and $(2,2)$ forms respectively, and both satisfy the two Bianchi identities. It follows that for all $j,k,\ell$ the double form $\sC_g^{j}(g^k R_g^{\ell})$ is symmetric, satisfies the first Bianchi identity, and, if $j=0,$ satisfies the second Bianchi identity.
\end{remark}

\begin{remark} The Hodge star extends to double forms by 
\begin{equation*}
	*(\alpha \otimes \beta) = (*\alpha) \otimes (*\beta).
\end{equation*}
A four-dimensional manifold is Einstein if and only if its curvature, as a $(2,2)$-form, satisfies $*R=R,$ so the 
Hitchin-Thorpe inequality \cite{Thorpe:GB, Hitchin:Einstein}
can be written
\begin{equation*}
	\text{ in $4$ dimensions, } *R=R \implies \chi(M) \geq \frac32|\mathrm{sign}(M)|,
\end{equation*}
where $\mathrm{sign}(M)$ denotes the signature of $M.$
Thorpe obtained this inequality as a particular instance of the more general
\begin{equation*}
	\text{ in $4k$ dimensions, } *R^k=R^k \implies \chi(M) \geq \frac{(k!)^2}{(2k)!}|\mathrm{p}_k(M)|,
\end{equation*}
where $\mathrm{p}_k(M)$ denotes the $k^{\text{th}}$ Pontrjagin number of the manifold.
Thorpe's higher dimensional self-dual metrics are Lovelock, see \cite{Labbi:OnGr} for a discussion and generalization, and seem natural objects to study.
\end{remark}

\begin{remark} \label{rmk:Effectives}
The Kulkarni-Nomizu product is most often encountered in the orthogonal decomposition of the curvature tensor
\begin{equation*}
	R = W + g\lrpar{\frac{\sC_g R-\tfrac{\sC_g^2R}m}{m-2}} + g^2 \lrpar{\frac{\sC_g^2 R}{2m(m-1)}}.
\end{equation*}
There is a similar decomposition of symmetric double forms satisfying the first Bianchi identity, such as
the double forms $R^k$ and their contractions, see \cite[\S3]{Kulkarni}.
\end{remark}

The following result will be very useful below (\cite[Lemma 2.1]{Labbi:pq}).

\begin{lemma} \label{lem:DoubleForms}
For any $\omega \in \Omega^{\ell\otimes\ell}(M)$ we have
\begin{align*}
	0)& 
	\sC_g(\omega g) = g  \sC(w) + (m-2\ell) \omega \\
	1)&
	\sC_g(g^k \omega) = g^k  \sC(w) + k (m-2\ell-k+1) g^{k-1} \omega \\
	2)&
	\sC_g^p(g^k\omega)
	= \sum_{r=0}^{p}
	\binom{m-2\ell+p-k}{r}
	\frac{k!}{(k-r)!}
	\frac{p!}{(p-r)!}
	g^{k-r} \sC_g^{p-r}(\omega)
%
\end{align*}
with the convention that if $k-r<0$ then $g^{k-r}=0$ and if $p-r \notin[0,\ell]$ then $\sC_g^{p-r}(\omega)=0.$
\end{lemma}

The same proof shows that for $\omega \in \Omega^{a\otimes b}(M),$
\begin{equation*}
	\sC_g^p(g^k\omega)
	= \sum_{r=0}^{p}
	\binom{m-a-b+p-k}{r}
	\frac{k!}{(k-r)!}
	\frac{p!}{(p-r)!}
	g^{k-r} \sC_g^{p-r}(\omega).
\end{equation*}

\begin{proof}
(0) is \cite[Proposition 2.4]{Kulkarni}.
We prove $(1)$ by induction, using (0) as our base case. The inductive step is
\begin{multline*}
	\sC(g^{k+1} \omega) = g \sC(g^k \omega) + (m-2(\ell+k)) g^k \omega \\
	= g^{k+1} \sC(\omega) + k(m-2\ell-k+1) g^k \omega
	+ (m-2(\ell+k)) g^k \omega \\
	= g^{k+1} \sC(\omega) + (k+1)(m-2\ell-k) g^k \omega.
\end{multline*}
Similarly we prove (2) by induction using (1) as our base case. The inductive step is, with $\bar m = m-2\ell,$
\begin{multline*}
	\sC_g^{p+1}(g^k\omega)
	= \sum_{r=0}^{p}
	\binom{\bar m+p-k}{r}
	\frac{k!}{(k-r)!}
	\frac{p!}{(p-r)!}
	\sC_g(g^{k-r} \sC_g^{p-r}(\omega))\\
	= \sum_{r=0}^{p}
	\binom{\bar m+p-k}{r}
	\frac{k!}{(k-r)!}
	\frac{p!}{(p-r)!}
	\Big(g^{k-r} \sC_g^{p+1-r}(\omega)
	+(k-r)(\bar m+2p-k-r+1)g^{k-r-1}\sC_g^{p-r}(\omega)\Big) \\
\begin{multlined}
	= \sum_{r=0}^{p}
	\frac{(\bar m+p-k)!}{r!(\bar m+p-k-r+1)!}\frac{p!}{(p-r+1)!}\frac{k!}{(k-r)!} \\
	\Big( (\bar m+p-k-r+1)(p-r+1)+r(\bar m+2p-k-r+2) \Big) g^{k-r}\sC^{p+1-r}(\omega)
\end{multlined} \\
	= \sum_{r=0}^{p+1}
	\frac{(\bar m+p-k)!}{r!(\bar m+p-k-r+1)!}\frac{p!}{(p-r+1)!}\frac{k!}{(k-r)!}
	\Big( (\bar m+p+1-k)(p+1) \Big) g^{k-r}\sC^{p+1-r}(\omega) \\
	= \sum_{r=0}^{p+1}
	\binom{\bar m+p+1-k}{r}
	\frac{k!}{(k-r)!}
	\frac{(p+1)!}{(p-r)!}
	g^{k-r} \sC_g^{p+1-r}(\omega)
\end{multline*}
\end{proof}

Some useful particular cases are 
\begin{equation*}
	\sC_g^k(g^k) = \frac{k!m!}{(m-k)!},
\end{equation*}
\begin{equation*}
	\sC_g^{k}(g^{k}\omega)
	= 
	\frac{(m-2)!k!}{(m-k-1)!}\lrpar{ k g\sC_g(\omega) + (m-k-1) \omega},
	\text{ whenever } \omega \in
	\Omega^{1\otimes 1}(M),
\end{equation*}
\begin{multline*}
	\sC_g^{k+1}(g^k\omega) = \frac{(m-3)!(k+1)!}{2(m-k-2)!}
	\lrpar{ k g \sC_g^2(\omega) + 2(m-k-2)\sC_g(\omega)},
	\text{ whenever } \omega \in 
	\Omega^{2\otimes 2}(M).
\end{multline*}
%

\section{Fefferman-Graham expansions} \label{sec:FG}

Let $(M,g)$ be a conformally compact manifold of dimension $m=n+1$ with curvature $R_g \in \Omega^{2\otimes2}(M).$
Recall that, for each $q < \frac{n+1}2,$
\begin{equation*}
	\cR^{(2q)}_g = \sC_g^{2q-1}R^q_g, \quad
	\scal^{(2q)}(g) = \sC_g^{2q}R^q_g, \quad
	\cE^{(2q)}_g = \cR^{(2q)}_g - \tfrac1{2q}\scal^{(2q)}(g) g.
\end{equation*}
For a hyperbolic metric ${\df h},$ using \eqref{eq:CurvCstSec}, these are given by 
\begin{equation*}
\begin{aligned}
	\cR^{(2q)}_{\df h} &= \lrpar{-\frac12}^q \frac{n!(2q)!}{(n-2q+1)!} {\df h} = \lambda^{(2q)} {\df h}, \\
	\scal^{(2q)}(\df h) &= \lrpar{-\frac12}^q \frac{(n+1)!(2q)!}{(n-2q+1)!} = (n+1)\lambda^{(2q)}, \\
	\cE^{(2q)}_{\df h} 
	&= \cR^{(2q)}_{\df h} - \frac{\scal^{(2q)}(\df h)}{2q}{\df h}
	= (1-\tfrac{n+1}{2q})\lambda^{(2q)}{\df h}.
\end{aligned}
\end{equation*}

In this section we follow \cite[\S2]{Graham:Vol} and work out the formal consequences of the equations
\begin{equation}\label{eq:PLEq}
	F_g(\alpha, \beta) = \sum \alpha_q (\cR^{(2q)}_g - \lambda^{(2q)}g) 
		+ \beta_q (\scal^{(2q)}(g)-(n+1)\lambda^{(2q)})g = 0,
\end{equation}
with $\alpha_q,$ $\beta_q$ constants (with the Lovelock equation corresponding to $\beta_q = -\alpha_q/(2q)$).

For given constants $\alpha,$ $\beta$ we define
\begin{multline*}
	\mathrm{LimSec}(\alpha, \beta) = \left\lbrace \kappa >0 :
	\sum \lambda^{(2q)} (\alpha_q +(n+1)\beta_q)(\kappa^q-1)=0, 
	\right. \\ \left.
	\Mand
	\bA_1(\alpha,\kappa) 
	= \sum \alpha_q \lrpar{-\frac\kappa2}^{q-1} \frac{(n-2)!}{2}\frac{(2q)!}{(n-2q)!} \neq 0
	\right\rbrace.
\end{multline*}
Note that since $(-1)^q\lambda^{(2q)}> 0,$ 
\begin{equation*}
	(-1)^q \alpha_q \geq 0 \text{ (or $\leq 0$)} \implies 1 \in \mathrm{LimSec}(\alpha, \beta),
\end{equation*}
and similarly,
\begin{equation*}
	(-1)^q (\alpha_q+(n+1)\beta_q) \geq 0 \text{ (or $\leq 0$)} \implies \mathrm{LimSec}(\alpha, \beta) \subseteq \{1\}.
\end{equation*}
On the other hand, by choosing $\alpha_q,$ $\beta_q,$ appropriately we can arrange 
\begin{equation*}
	\sum \lambda^{(2q)} (\alpha_q +(n+1)\beta_q)(\kappa^q-1) = (\kappa-1)p(\kappa)
\end{equation*}
for any polynomial $p$ of degree $\floor{\frac{n+1}2}-1,$ and so we can arrange for there to be $\floor{\frac{n+1}2}$ different positive solutions $\kappa.$

\begin{remark}
For concreteness, if 
\begin{equation*}
	\alpha = (6n(n-2)(n-3)a, 1, 0, \ldots, 0), \quad
	\beta_q = -\frac{\alpha_q}{2q}
\end{equation*}
then we are studying the equation
\begin{equation*}
	6(n-2)(n-3)a E^{(2)}_{ij}(g) + E^{(4)}_{ij}(g) 
	= 6n(n-1)(n-2)(n-3)\lrpar{a-1}g_{ij},
\end{equation*}
which is satisfied by any hyperbolic metric. The set $\mathrm{LimSec}(\alpha, \beta)$ in this case is defined as those $\kappa>0$ satisfying
\begin{equation*}
\begin{gathered}
	-\tfrac32n(n-1)(n-2)(n-3)(\kappa-1)(\kappa+1-2a) =0\\
	\bA_1(\alpha,\kappa) = 6(n-2)(n-3)(a - \kappa) \neq 0
\end{gathered}
\end{equation*}
and hence, for this particular choice of $(\alpha, \beta),$ we have
\begin{equation*}
	\mathrm{LimSec}(\alpha, \beta) = 
	\begin{cases}
	\emptyset & \Mif a=1 \\
	\{1, 2a-1\} & \Mif a>\tfrac12, a\neq 1 \\
	\{1\} & \Motherwise
	\end{cases}
\end{equation*}
\end{remark}
$ $\\

During the computations below, {\bf we will assume}
\begin{equation*}
	\alpha \neq -(n+1)\beta, \quad
	\mathrm{LimSec}(\alpha,\beta) \neq \emptyset.
\end{equation*}
We will make use of the following functions to simplify the expressions we obtain,
\begin{equation*}
\begin{aligned}
	\bA_2(\alpha, \kappa) 
	&= \sum \alpha_q \lrpar{-\frac\kappa2}^{q-1} \frac{(n-2)!}{2}\frac{(2q)!}{(n-2q+1)!} (q-1) \\
	\bA_3(\alpha,\kappa) 
	&= \sum \alpha_q \lrpar{-\frac\kappa2}^{q-2} \frac{(n-4)!}{4!}\frac{(2q)!}{(n-2q)!} (q-1) \\
	\bA_4(\alpha, \kappa) 
	&= \sum \alpha_q \lrpar{-\frac\kappa2}^{q-2} \frac{(n-4)!}{4!}\frac{(2q)!}{(n-2q+1)!} \frac{(q-1)(q-2)}2 \\
	\bB_{i,j}(\alpha, \beta, \kappa)
	&= \bA_j(\alpha)+\bA_i(\beta)+(n+1)\bA_j(\beta)
\end{aligned}
\end{equation*}
and we point out that for the usual Einstein equation $\Ric(g) = -n g$ we have $\alpha = (1,0,\ldots,0),$ $\beta=0,$ $\mathrm{LimSec}=\{1\},$
\begin{equation*}
	\lambda(\alpha) = -n, \quad
	\bA_1(\alpha,1) =1, \quad
	\bA_2(\alpha,1) =
	\bA_3(\alpha,1)=
	\bA_4(\alpha,1)=0.
\end{equation*}
%

\subsection{Asymptotically hyperbolic}

First, when does \eqref{eq:PLEq} imply that $g$ is asymptotically hyperbolic?

Let $x$ be a boundary defining function, i.e., a non-negative function smooth on $\bar M$ with $\pa M = \{x=0\}$ and vanishing to first order at $\pa M,$ and let 
\begin{equation}\label{eq:AH}
	\kappa = |\tfrac{dx}x|^2_g\rest{\pa M}.
\end{equation}
Mazzeo \cite[pg. 311]{Mazzeo:Hodge}
pointed out that the curvature of $g$ satisfies
\begin{equation*}
	R_g = -\kappa \tfrac{g^2}2 + \cO(x^{-3}) \Mas x \to 0.
\end{equation*}
It follows that
\begin{equation*}
	\sC_g^{2q}R_g^q = \kappa^q (n+1)\lambda^{(2q)} + \cO(x)
\end{equation*}
and hence 
\begin{equation*}
	F_g(\alpha, \beta)
	= \sum \lambda^{(2q)} (\alpha_q +(n+1)\beta_q)(\kappa^q-1)g +\cO(x^{-1}).
\end{equation*}
Thus $F_g(\alpha, \beta)=0$ implies 
\begin{equation}\label{eq:PolyKappa}
	\sum \lambda^{(2q)} (\alpha_q +(n+1)\beta_q)(\kappa^q-1)=0.
\end{equation}

We conclude that as long as  $\alpha_q+(n+1)\beta_q\not\equiv 0,$ which are assuming, the sectional curvatures of $g$ converge to a locally constant function as $x \to 0.$
For the computations below, $\kappa$ can be any locally constant function valued in $\mathrm{LimSec}(\alpha,\beta).$

\subsection{General expansion}
Now that we know that sectional curvatures of $g$ converge to $\kappa,$ a locally constant function on $\pa M,$ 
it follows from \cite[Lemma 2.1]{Graham:Vol} that for any boundary defining function $x_0$ there is 
another boundary defining function $x$ such that
\begin{equation*}
	x = x_0 + \cO(x_0^2) \Mas x_0 \to 0, \quad
	|\tfrac{dx}x|^2_g \equiv \kappa \text{ in a neighborhood of }\pa M.
\end{equation*}
Boundary defining functions satisfying the latter condition are known as {\em special}, or geodesic, boundary defining functions. 

From now on we assume that $x$ is a special boundary defining function, and we introduce the notation
\begin{equation*}
	\bar g = x^2 g, \quad h_0 = \bar g\rest{\pa M},
\end{equation*}
for the associated incomplete metric and boundary metric, respectively.
We use the integral curves of $\nabla_{\bar g}x$ to identify a neighborhood of $\pa M$ with a collar $[0,1)_x \times \pa M$ in which the metric takes the form
\begin{equation*}
	\bar g = \frac{dx^2}{\kappa} + h,
\end{equation*}
and we will work in this neighborhood (cf. \cite[Lemma 5.2]{Graham-Lee}).

In this neighborhood, the curvature of $g$ satisfies 
\begin{align*}
	g(R_g(\pa_x, \pa_i)\pa_x,\pa_j)
	&= x^{-4}(
	-\tfrac12 x^2h''_{ij}
	+\tfrac14 x^2 h'_{ia}h^{ab}h'_{bj}
	+ \tfrac12 x h'_{ij}
	-h_{ij} ) \\
	g(R_g(\pa_i, \pa_j)\pa_k,\pa_\ell)
	&= x^{-4}
	\Big(x^2
	h(R_h(\pa_i, \pa_j)\pa_k,\pa_\ell)
	-\tfrac{x^2\kappa}4(h'_{ik}h'_{\ell j}-h'_{i\ell}h'_{kj}) \\
	&\phantom{xxxxxxx}
	+\tfrac{x\kappa}2(h_{ik}h'_{\ell j}+h'_{ik}h_{\ell j}-h_{i\ell}h'_{kj}-h_{i\ell}h'_{kj}) 
	-\kappa(h_{ik}h_{\ell j}-h_{i\ell}h_{kj})\Big)\\
	g(R_g(\pa_x, \pa_i)\pa_j,\pa_k)
	&= \frac1{2x^2}((\nabla_{\pa_k}h')(\pa_i, \pa_j)
	- (\nabla_{\pa_j}h')(\pa_i, \pa_k) )
\end{align*}
We can reexpress this, using \eqref{eq:Example}, as an equality of $(2,2)$-forms
\begin{multline}\label{eq:R2form}
	R_g = (dx \otimes dx) \owedge
	x^{-4}(
	\tfrac{x^2}2(-h''+\tfrac12\sC_h(h')h'-\tfrac14\sC_h((h')^2))
	+ \tfrac x2  h' -h ) 
	\\
	+ x^{-2}\cS((dx\otimes 1)\owedge Dh')
	+ x^{-4}(x^2 R_h-\tfrac\kappa2(\tfrac x2h'-h)^2)
\end{multline}
where $Dh'$ is the $(1,2)$ form 
\begin{equation*}
	Dh'(U)(V,W) = \frac12 ( (\nabla_Vh')(U,W) - (\nabla_Wh')(U,V))
\end{equation*}
and $\cS((dx\otimes 1)\owedge Dh')$ denotes the symmetric $(2,2)$ form extending $(dx\otimes 1)\owedge Dh'.$

Taking $q^{\text{th}}$ power, we see that $R_g^q$ is given by
\begin{multline*}
	R_g^q
	= x^{-4q}\Big(
	q(dx \otimes dx) \owedge
	(
	\tfrac{x^2}2(-h''+\tfrac12\sC_h(h')h'-\tfrac14\sC_h((h')^2))
	+ \tfrac x2  h' -h ) 
	(x^2 R_h-\tfrac\kappa2(\tfrac x2h'-h)^2)^{q-1} \\
	+qx^{2}\cS((dx \otimes 1)\owedge Dh' (x^2 R_h-\tfrac\kappa2(\tfrac x2h'-h)^2)^{q-1})
	+(x^2 R_h-\tfrac\kappa2(\tfrac x2h'-h)^2)^q \Big),
\end{multline*}
its $(2q-1)^{\text{th}}$ contraction by
\begin{multline}\label{eq:R2q}
	\sC_g^{2q-1}(R_g^q) = \\
	x^{-2}\Big(
	q(dx \otimes dx) \owedge
	\sC_h^{2q-1}\Big((
	\tfrac{x^2}2(-h''+\tfrac12\sC_h(h')h'-\tfrac14\sC_h((h')^2))
	+ \tfrac x2  h' -h ) 
	(x^2 R_h-\tfrac\kappa2(\tfrac x2h'-h)^2)^{q-1}\Big) \\
	+qx^{2}\cS((dx \otimes 1)\owedge \sC_h^{2q-1}(Dh' (x^2 R_h-\tfrac\kappa2(\tfrac x2h'-h)^2)^{q-1}))\\
	+\sC_h^{2q-1}((x^2 R_h-\tfrac\kappa2(\tfrac x2h'-h)^2)^q) \\
	+ (2q-1)q\kappa  
	\sC_h^{2q-2}\Big((
	\tfrac{x^2}2(-h''+\tfrac12\sC_h(h')h'-\tfrac14\sC_h((h')^2))
	+ \tfrac x2  h' -h ) 
	(x^2 R_h-\tfrac\kappa2(\tfrac x2h'-h)^2)^{q-1}\Big) \Big),
\end{multline}
and its $(2q)^{\text{th}}$ contraction by
\begin{multline}\label{eq:ell2q}
	\sC_g^{2q}(R_g^q) 
	= 
	\sC_h^{2q}((x^2 R_h-\tfrac\kappa2(\tfrac x2h'-h)^2)^q) \\
	+ (2q)q\kappa  
	\sC_h^{2q-1}\Big((
	\tfrac{x^2}2(-h''+\tfrac12\sC_h(h')h'-\tfrac14\sC_h((h')^2))
	+ \tfrac x2  h' -h ) 
	(x^2 R_h-\tfrac\kappa2(\tfrac x2h'-h)^2)^{q-1}\Big) \Big).
\end{multline}

A priori, $F_g(\alpha, \beta)$ is $\cO(x^{-2}),$ but as we saw above the most singular term in $\cR^{(2q)}_g$ cancels with that in $\lambda^{(2q)}g$ and the most singular term in $\scal^{(2q)}(g)$ with $(n+1)\lambda^{(2q)}.$ Thus the most singular term in $F_g(\alpha, 0)$ is 
\begin{equation*}
\begin{gathered}
\begin{multlined}
	x^{-1}\sum \alpha_q \Big(
	(dx \otimes dx) \owedge
	\sC_h^{2q-1} ( q(q-\tfrac12)h'(-\tfrac\kappa2h^2)^{q-1}) \\
	+(2q-1)\kappa
	\sC_h^{2q-2} ( q(q-\tfrac12)h'(-\tfrac\kappa2h^2)^{q-1})
	- \sC_h^{2q-1} (q h'(-\tfrac\kappa2h^2)^{q})
	\Big) 
\end{multlined} \\
\begin{multlined}
	=x^{-1}\sum \alpha_q \Big(
	(dx \otimes dx) \owedge
	\frac{2q-1}2\lrpar{-\frac{\kappa}2}^{q-1}\frac{(n-1)!}2\frac{(2q)!}{(n-2q+1)!}\sC_h(h') \\
	- \lrpar{-\frac\kappa2}^q\frac{(n-1)!}2\frac{(2q)!}{(n-2q+1)!}
	( (n-2q+1)h' + (2q-1) h \sC_h(h') )
	\Big)
\end{multlined} \\
\begin{multlined}
	=x^{-1} \Big(
	(dx \otimes dx) \owedge \frac12(\bA_1(\alpha, \kappa)+2n\bA_2(\alpha, \kappa))\sC_h(h') \\
	+ \frac\kappa2 ((n-1)\bA_1(\alpha, \kappa) h' + (\bA_1(\alpha, \kappa)+2n\bA_2(\alpha, \kappa))h \sC_h(h'))
	\Big),
\end{multlined}
\end{gathered}
\end{equation*}
and so the most singular term in $F_g(\alpha, \beta)$ is
\begin{equation*}
\begin{gathered}
	x^{-1} \Big(
	(dx \otimes dx) \owedge \frac12(\bA_1(\alpha, \kappa)+2n\bA_2(\alpha, \kappa))\sC_h(h') \\
	+ \frac\kappa2 ((n-1)\bA_1(\alpha, \kappa) h' + (\bA_1(\alpha, \kappa)+2n\bA_2(\alpha, \kappa))h \sC_h(h'))
	\Big) \\
	+
	x^{-1}\kappa n(\bA_1(\beta, \kappa)+(n+1)\bA_2(\beta, \kappa))\sC_h(h')(\frac{dx \otimes dx}\kappa+ h).
\end{gathered}
\end{equation*}

The equation $F_g(\alpha, \beta)=0$ imposes that both the coefficient of $dx \otimes dx$ and the complement  vanish at $x=0,$
\begin{equation*}
\begin{gathered}
	\Big(\bA_1(\alpha, \kappa)+2n\bA_2(\alpha, \kappa)
	+  2n(\bA_1(\beta, \kappa)+(n+1)\bA_2(\beta, \kappa)) \Big)\sC_h(h') = \cO(x), \\
\begin{multlined}
	\Big((n-1)\bA_1(\alpha, \kappa) h' + (\bA_1(\alpha, \kappa)+2n\bA_2(\alpha, \kappa) 
		+ 2n(\bA_1(\beta, \kappa)+(n+1)\bA_2(\beta, \kappa)))h \sC_h(h'))\Big) \\
	=\cO(x).
\end{multlined}
\end{gathered}
\end{equation*}
Substituting the first equation into the second yields $h' = \cO(x)$ as long as $\bA_1(\alpha, \kappa)\neq 0.$

From \eqref{eq:R2q} we see that the terms with a factor of $x^0$ in $\sum \alpha_q \cR_g^{(2q)}$ are
\begin{equation*}
\begin{gathered}
\begin{multlined}
	\sum \alpha_q \Big(
	(dx \otimes dx) \owedge
	\sC_h^{2q-1} ( 
	-\tfrac q2(-\tfrac\kappa2)^{q-1}h^{2q-2}h'' - q(q-1)(-\tfrac\kappa2)^{q-2}h^{2q-3}R_h
	) \\
	+q\cS((dx\otimes1)\owedge \sC_h^{2q-1}(Dh' (-\tfrac\kappa2 h^2)^{q-1}))\\
	+(2q-1)\kappa
	\sC_h^{2q-2} ( 
	-\tfrac q2(-\tfrac\kappa2)^{q-1}h^{2q-2}h'' - q(q-1)(-\tfrac\kappa2)^{q-2}h^{2q-3}R_h
	) \\
	+ \sC_h^{2q-1} (q(-\tfrac\kappa2)^{q-1}h^{2q-2}R_h)
	\Big)
\end{multlined} \\
\begin{multlined}
	=\sum \alpha_q \Big(
	(dx \otimes dx) \owedge
	\Big( -\frac 12 \lrpar{-\frac\kappa2}^{q-1} \frac{(2q)!}{(n-2q+1)!} \frac{(n-1)!}2 \sC_h(h'') \\
	-(q-1)\lrpar{-\frac\kappa2}^{q-2}\frac{(2q)!}{(n-2q+1)!}\frac{(n-2)!}2\sC_h^2(R_h)\Big) \\
	+\lrpar{-\frac\kappa2}^{q-1} \frac{(2q)!}{(n-2q)!} \frac{(n-2)!}2
	\cS((dx\otimes1)\owedge \sC_h(Dh'))\\
	+
	\lrpar{-\frac\kappa2}^{q} \frac{(2q)!}{(n-2q+1)!}\frac{(n-2)!}2
	( (n-2q+1)h'' + (2q-2)h \sC_h(h'')) \\
	+\lrpar{-\frac\kappa2}^{q-1} \frac{(2q)!}{(n-2q+1)!}\frac{(n-2)!}2
	((n-2q+1)\sC_h(R_h) + (q-1)h\sC_h^2(R_h))
	\Big)
\end{multlined} \\
\begin{multlined}
	=(dx \otimes dx) \owedge
	\Big( -\frac12(\bA_1(\alpha, \kappa) + 2\bA_2(\alpha, \kappa)) \sC_h(h'') 
	+\frac2\kappa\bA_2(\alpha, \kappa)\sC_h^2(R_h) \Big) \\
	+\bA_1(\alpha, \kappa)
	\cS((dx\otimes1)\owedge \sC_h(Dh'))\\
	+
	\lrpar{-\frac\kappa2}
	( \bA_1(\alpha, \kappa)h'' + 2\bA_2(\alpha, \kappa)h \sC_h(h'')) 
	+
	(\bA_1(\alpha, \kappa)\sC_h(R_h) + \bA_2(\alpha, \kappa)h\sC_h^2(R_h)).
\end{multlined}
\end{gathered}
\end{equation*}
Hence the terms in $F_g(\alpha, \beta)$ with a factor of $x^0$ are
\begin{equation*}
\begin{gathered}
	(dx \otimes dx) \owedge
	\Big( -\frac12(\bA_1(\alpha, \kappa) + 2\bA_2(\alpha, \kappa)) \sC_h(h'') 
	+\frac2\kappa\bA_2(\alpha, \kappa)\sC_h^2(R_h) \Big) \\
	+\bA_1(\alpha, \kappa)
	\cS((dx\otimes1)\owedge \sC_h(Dh'))\\
	+
	\lrpar{-\frac\kappa2}
	( \bA_1(\alpha, \kappa)h'' + 2\bA_2(\alpha, \kappa)h \sC_h(h'')) 
	+
	(\bA_1(\alpha, \kappa)\sC_h(R_h) + \bA_2(\alpha, \kappa)h\sC_h^2(R_h)) \\
\begin{multlined}
	+\Big(
	-\kappa(\bA_1(\beta, \kappa)+(n+1)\bA_2(\beta, \kappa))\sC_h^2(h'')
	+(\bA_1(\beta, \kappa)+(n+2)\bA_2(\beta, \kappa)) \sC_h^2(R_h)\Big) \\
	( \frac{dx\otimes dx}\kappa + h).
\end{multlined}
\end{gathered}
\end{equation*}

Since $h'=\cO(x),$ we can use \eqref{eq:R2q}, \eqref{eq:ell2q} to write $F_g(\alpha, \beta)=0$ as
\begin{multline}\label{eq:SR2q}
	(dx \otimes dx) \owedge \frac12(\bA_1(\alpha, \kappa)+2n\bA_2(\alpha, \kappa))\sC_h(h') \\
	+ \frac\kappa2 ((n-1)\bA_1(\alpha, \kappa) h' + (\bA_1(\alpha, \kappa)+2n\bA_2(\alpha, \kappa))h \sC_h(h'))
	\\
	+
	\kappa n(\bA_1(\beta, \kappa)+(n+1)\bA_2(\beta, \kappa))\sC_h(h')(\frac{dx \otimes dx}\kappa+ h) \\
	+x\Big(
	(dx \otimes dx) \owedge
	\Big( -\frac12(\bA_1(\alpha, \kappa) + 2\bA_2(\alpha, \kappa)) \sC_h(h'') 
	+\frac2\kappa\bA_2(\alpha, \kappa)\sC_h^2(R_h) \Big) \\
	+\bA_1(\alpha, \kappa)
	\cS((dx\otimes1)\owedge \sC_h(Dh'))\\
	+
	\lrpar{-\frac\kappa2}
	( \bA_1(\alpha, \kappa)h'' + 2\bA_2(\alpha, \kappa)h \sC_h(h'')) 
	+
	(\bA_1(\alpha, \kappa)\sC_h(R_h) + \bA_2(\alpha, \kappa)h\sC_h^2(R_h)) \\
	+\Big(
	-\kappa(\bA_1(\beta, \kappa)+(n+1)\bA_2(\beta, \kappa))\sC_h^2(h'')
	+(\bA_1(\beta, \kappa)+(n+2)\bA_2(\beta, \kappa)) \sC_h^2(R_h)\Big)( \frac{dx\otimes dx}\kappa + h) \Big) \\
	= \cO(x^3).
\end{multline}
Taking $k$ derivatives with respect to $x$ we find
\begin{multline}\label{eq:PLk}
	(dx \otimes dx) \owedge \frac{1}2((1-k)\bA_1(\alpha, \kappa)+2(n-k)\bB_{1,2}(\alpha, \beta, \kappa) )\sC_h(h^{(k+1)}) \\
	+ \frac\kappa2 (n-1-k)\bA_1(\alpha, \kappa) h^{(k+1)} \\
	+ \frac\kappa2(\bA_1(\alpha, \kappa)+2(n-k)\bB_{1,2}(\alpha, \beta, \kappa))h \sC_h(h^{(k+1)}))  \\
	= \text{ terms involving fewer derivatives of $h$ } + \cO(x).
\end{multline}
Restricting the coefficient of $dx \otimes dx$ and the contraction of the coefficient without $dx$ to $x=0$ yield the equations
\begin{equation*}
\begin{gathered}
\begin{multlined}
	((1-k)\bA_1(\alpha, \kappa)+2(n-k)\bB_{1,2}(\alpha, \beta, \kappa) )\sC_h(h^{(k+1)})|_{x=0} \\
	\phantom{filler for formating xxx}
	= \text{ terms involving fewer derivatives of $h$ } + \cO(x), 
\end{multlined} \\
\begin{multlined} 
	\Big((n-1-k)\bA_1(\alpha, \kappa) 
	+ n(\bA_1(\alpha, \kappa)+2(n-k)\bB_{1,2}(\alpha, \beta, \kappa)) \Big)
	 \sC_h(h^{(k+1)}))|_{x=0} \\
	= \text{ terms involving fewer derivatives of $h$ } + \cO(x).
\end{multlined}
\end{gathered}
\end{equation*}
Note that if $\bA_1(\alpha, \kappa)\neq0$ then the two coefficients of $\sC_h(h^{(k+1)})|_{x=0}$ can not both be zero;
indeed if the first should vanish, then the second can be written as
$(n-1)(k+1)\bA_1(\alpha, \kappa).$
Hence if we have determined $\{h_0, h'|_{x=0}, \ldots, h^{(k)}|_{x=0}\}$ we can determine $\sC_{h_0}(h^{(k+1)})|_{x=0}$ and then, as long as $k+1\neq n,$ use  \eqref{eq:PLk} to determine $h^{(k+1)}|_{x=0}.$

It follows inductively that the equations
\begin{equation}\label{eq:Constructive}
	\text{ on-diagonal parts of }F_g(\alpha, \beta) = \cO(x^{n-2}), 
	\quad
	\sC_g(F_g(\alpha, \beta)) = \cO(x^{n-1})
\end{equation}
(note that the analysis above only involved the on-diagonal parts of $F_g(\alpha,\beta)$ with respect to the splitting $\ang{\pa_x}\oplus\ang{\pa_x}^{\perp}$)
uniquely determine a metric, up to order $x^{n-2},$ of the form $x^{-2}(dx^2/\kappa + h)$ with 
$h^{(\ell)}|_{x=0}$ for $\ell<n,$ and $\sC_{h_0}(h^{(n)})|_{x=0},$ natural tensor invariants of $h|_{x=0},$ and, since the left hand side of \eqref{eq:PLEq} respects parity in $x$ (see \eqref{eq:R2form}), with $h^{(\ell)}|_{x=0}=0$ for $\ell<n$ odd. \\

When $\ell=n,$  equation \eqref{eq:PLk} is
\begin{multline*}
	(dx \otimes dx) \owedge \frac{1}2( (2-n)\bA_1(\alpha, \kappa)+2\bB_{1,2}(\alpha, \beta, \kappa))\sC_{h_0}(h^{(n)})  \\
	+ \frac\kappa2 (\bA_1(\alpha, \kappa) + 2\bB_{1,2}(\alpha, \beta, \kappa))h_0 \sC_{h_0}(h^{(n)})  \\
	= \text{ terms involving fewer derivatives of $h$ } + \cO(x).
\end{multline*}
If $n$ is odd, then by parity the right hand side is $\cO(x),$ so $ \sC_{h_0}(h^{(n)})|_{x=0}=0,$ but the trace-free part is unconstrained.
If $n$ is even, then the right hand side may have a non-vanishing trace-free part, so that the expansion of $h$ must include a term $x^n \log x$ with a trace-free coefficient. \\

In this way we have shown Theorem \ref{thm:FefGrExp} that metrics satisfying $F_g(\alpha,\beta)=0$ formally have a Fefferman-Graham expansion. 
On the other hand, if we start with $h_0,$ we have only shown how to arrange \eqref{eq:Constructive}.
Following \cite{Graham-Hirachi:Obst} we next show that in the particular case of the Lovelock equations, i.e., when $\beta_q = -\alpha_q/2q$ so that $F_g(\alpha,\beta)$ is a linear combination of Lovelock tensors, the off-diagonal terms are related to the on-diagonal because the Lovelock tensors are divergence-free.

\begin{lemma}\label{lem:DivFree}
If $\beta_q = -\alpha_q/(2q),$ and $g$ satisfies \eqref{eq:Constructive}, then $g$ satisfies
\begin{equation*}\label{eq:Constructive'}
	\text{ off-diagonal parts of }F_g(\alpha, \beta) = \cO(x^{n-1}),
\end{equation*}
and, if $\sO = x^{2-n}\;\text{trace-free}(F_g(\alpha,\beta))|_{x=0},$ then  the divergence of $\sO$ is determined by $h_0$ and vanishes if $n$ is odd.

\end{lemma}

\begin{proof}
We  compute in local coordinates, where $\pa_0=\pa_x,$ indices $\{s, t, u, v\}$ vary in $\{0, \ldots, n\}$ and indices $\{i,j,k,\ell\}$ vary in $\{1, \ldots, n\}.$

If $F_{ij}$ is the $(0,2)$-tensor corresponding to $F_g(\alpha, \beta),$ then it satisfies
\begin{equation*}
	0 = g^{st} F_{tu;s} 
	= g^{st}(\pa_sF_{tu} - \Gamma^v_{st}F_{vu}-\Gamma_{su}^vF_{tv}),  \Mforall u \in \{0,\ldots, n\}.
\end{equation*}
For a metric of the form $g = x^{-2}(dx^2/\kappa + h_x)$ the Christoffel symbols satisfy
\begin{equation*}
	\Gamma^k_{ij} = \bar \Gamma^k_{ij}, \quad
	\Gamma^0_{jk} = -\frac\kappa2 h'_{jk} + \frac\kappa x h_{jk}, \quad
	\Gamma^i_{0k} = \frac12 h^{i\ell}h'_{\ell k} - \frac1x \delta^i_k, \quad
	\Gamma^0_{0k} = \Gamma^i_{00} = 0, \quad
	\Gamma^0_{00} = -\frac1x,
\end{equation*}
where $\bar\Gamma$ denotes the Christoffel symbol of $h_x.$
Hence we have
\begin{multline*}
	0 = g^{st}F_{tu;s}
	= x^2 \left[
	\kappa \lrpar{ \pa_0F_{0u} + \frac1x F_{0u} + \frac{\delta_{0u}}xF_{00} }
	\right. \\ \left.
	+h^{ij}\pa_i F_{ju} + \kappa \lrpar{\frac12 h^{ij}h_{ij}'-\frac{n}x}F_{0u}
	-h^{ij}\bar\Gamma^k_{ij} F_{ku}
	-h^{ij}\Gamma_{iu}^k F_{jk} \right].
\end{multline*}
For $u=0$ this says
\begin{equation*}
	\kappa \lrpar{ \pa_x + \frac{2-n}x + \frac12 h^{ij}h_{ij}' } F_{00}  
	=
	\lrpar{\frac12h^{ij}h^{k\ell}h'_{i\ell}-h^{jk}} F_{jk}
	-
	(h^{ij}\pa_i  
	-h^{ik}\bar\Gamma^j_{ik} )F_{j0}
\end{equation*}
and for $u=\ell \neq 0,$ this says
\begin{equation}\label{eq:DivNot0}
	\kappa\lrpar{\pa_x + \frac{1-n}x + \frac12 h^{ij}h_{ij}'}F_{0\ell}
	= -h^{ij}F_{j\ell;i}.
\end{equation}
From \eqref{eq:R2q} we know that an expansion in $x$ for $h$ induces an expansion in $x$ for $F_{0\ell},$ starting at $x^0.$

Now since the right hand side of  \eqref{eq:DivNot0} is $\cO(x^{n-2})$ and $\frac12 h^{ij}h_{ij}' = \cO(x),$ if we assume that 
$F_{0\ell} = A_s x^s + B_s x^s\log x + \cO(x^{s+1})$ then we have
\begin{equation*}
	(s+1-n) (A_s x^{s-1} + B_s x^{s-1}\log x) 
	+ B_s x^{s-1} + \cO(x^{s})
	= \cO(x^{n-2}),
\end{equation*}
hence $A_s=B_s=0$ if $s<n-1.$
On the other hand if $s=n-1$ and we write $F_{j\ell} = \sO_{j\ell} x^{n-2} + \cO(x^{n-1}),$ then this same equation tells us that $h_0^{ij}\sO_{j\ell;i}=B_{n-1}.$
Thus we can conclude that $F_{0\ell}= \cO(x^{n-1}),$ and that the divergence of $\sO$ is determined by $h_0$ and vanishes if $n$ is odd.
\end{proof}

This finishes the proof of part (b) of Theorem \ref{thm:FefGrExp}. We have already shown parts (i) and (ii) for more general equations $F_g(\alpha, \beta)=0,$ and part (iii) follows from Lemma \ref{lem:DivFree}.

If $n$ is even then this finishes the proof of part (a) of Theorem \ref{thm:FefGrExp}, since we now know that any smooth conformally compact metric whose Taylor expansion is as above satisfies 
\begin{equation*}
	\sum \alpha_q E^{(2q)}_{ij}(g) = \lambda(\alpha) g_{ij} + \cO(x^{n-2}).
\end{equation*}
As in \cite{Graham-Hirachi:Obst} this also shows that it is unique modulo $\cO(x^{n-2})$ up to a diffeomorphism fixing $X \times \{0\}.$

If $n$ is odd and we postulate that $h$ has a smooth Taylor expansion at $\pa M,$
\begin{equation*}
	h \sim \sum x^j h^{(j)},
\end{equation*}
then the first term that we have yet to determine, $x^n h^{(n)},$ is also the first odd power of $x.$
We know that the on-diagonal part of the equation $F_g(\alpha,\beta)=0$ only imposes $\sC_g(h^{(n)})=0,$
and the argument in Lemma \ref{lem:DivFree} tells us that the off-diagonal parts of $F_g(\alpha, \beta)$ are $\cO(x^{n-1})$ but does not determine the $x^{n-1}$ term.
Fortunately, since this is the first odd power of $x$ in the expansion of $h,$ it is easy to determine directly from \eqref{eq:R2q} that 
the $x^{n-1}$ term in the expansion of the off-diagonal part of $F_g(\alpha, \beta)$ is
\begin{equation*}
	\sum \alpha_q \lrpar{
	q\cS((dx \otimes 1)\owedge \sC_h^{2q-1}(Dh^{(n)} (h^{2q-2}) ) ) }.
\end{equation*}
In particular, if we set $h^{(n)}=0$ then the off-diagonal part of $F_g(\alpha, \beta)$ is $\cO(x^{n}).$

After setting $h^{(n)}=0$ we can now continue as above and use the on-diagonal parts of the equation 
$F_g(\alpha,\beta)=0$ to determine the full Taylor expansion of $h$ (involving only even powers of $x$) 
and then use Lemma \ref{lem:DivFree} to see that  the $h$ we have constructed satisfies
\begin{equation*}
	F_g(\alpha, \beta) = \cO(x^{\infty})
\end{equation*}
(but we emphasize that here $\beta_q = -\alpha_q/(2q)$).
This finishes the proof of part (a) of Theorem \ref{thm:FefGrExp} when $n$ is odd.

\subsection{First couple of terms} \label{sec:FirstCouple}

A Poincar\'e-Lovelock metric in dimension greater than four has an expansion
\begin{equation*}
	g = x^{-2}(dx^2 + h_0 + x^2h_2 + x^4h_4 + \cO(x^5))
\end{equation*}
with $h_2$ and $h_4$ symmetric two tensors on $\pa M$ locally determined from $h_0.$
In this subsection we follow \cite[\S6.9]{Juhl:Book} and determine the coefficients $h_2$ and $h_4$ in terms of $h_0.$
We will show that $h_2$ is always a constant multiple of the Schouten tensor of $h_0,$ while $h_4$ depends on the coefficients of \eqref{eq:PLEq}.

In a technique used in {\em loc. cit.} but that goes back at least to \cite{Henningson-Skenderis}, 
we introduce the coordinate $\rho = x^2$ so that the metric takes the form
\begin{equation*}
	g = \frac{d\rho^2}{4\kappa \rho^2} + \frac{h}\rho.
\end{equation*}
In this subsection we consider $h$ as a function of $\rho$ and we will use $\dot h$ to denote $\pa_{\rho}h.$

We obtain an expression for the curvature in these coordinates from \eqref{eq:R2form} by making the replacements
\begin{equation*}
	x\mapsto \rho^{1/2}, \quad
	dx \mapsto  \frac{d\rho}{2\rho^{1/2}}, \quad
	\pa_x h \mapsto 2\rho^{1/2} \pa_{\rho}h, 
\end{equation*}
which yields
\begin{equation*}
\begin{gathered}
	R_g 
	=
	(d\rho\otimes d\rho)\owedge
	\lrpar{ \frac1{4\rho} (-2\ddot{h}+\sC_h(\dot{h})\dot{h}-\tfrac12\sC_h(\dot{h}^2)) - \frac1{4\rho^3}h}\\
	+\frac1\rho \cS((d\rho\otimes 1) \owedge D(\dot{h}))
	+\frac1\rho R_h
	-\frac\kappa{2\rho^2}(\rho \dot{h}-h)^2.
\end{gathered}
\end{equation*}
Similarly,
\begin{multline*}
	\sC_g^{2q-1}(R_g^q) \\
	=
	(d\rho\otimes d\rho)\owedge \rho^{-2}\sC_h^{2q-1}\Big[
	\frac{q}{4}\Big( \rho^2 (-2\ddot{h}+\sC_h(\dot{h})\dot{h}-\tfrac12\sC_h(\dot{h}^2)) -h \Big)
	 \Big( -\tfrac\kappa2 \rho^2\dot{h}^2 +\rho(R_h +\kappa h\dot{h}) -\tfrac\kappa2 h^2\Big)^{q-1}\Big]\\
	+q \cS\Big((d\rho\otimes 1) \owedge \sC_h^{2q-1}\Big[
		D(\dot{h})
		 \Big( -\tfrac\kappa2 \rho^2\dot{h}^2 +\rho(R_h +\kappa h\dot{h}) -\tfrac\kappa2 h^2\Big)^{q-1} \Big]\Big) \\
	+\rho^{-1}\sC_h^{2q-1}
	\Big( -\tfrac\kappa2 \rho^2\dot{h}^2 +\rho(R_h +\kappa h\dot{h}) -\tfrac\kappa2 h^2\Big)^{q} \\
	+(2q-1)q\kappa\rho^{-1}\sC_h^{2q-2}\Big[
	\Big( \rho^2 (-2\ddot{h}+\sC_h(\dot{h})\dot{h}-\tfrac12\sC_h(\dot{h}^2)) -h \Big)
	 \Big( -\tfrac\kappa2 \rho^2\dot{h}^2 +\rho(R_h +\kappa h\dot{h}) -\tfrac\kappa2 h^2\Big)^{q-1}\Big],
\end{multline*}
and
\begin{multline*}
	\sC_g^{2q}(R_g^q) 
	=
	\sC_h^{2q}
	\Big( -\tfrac\kappa2 \rho^2\dot{h}^2 +\rho(R_h +\kappa h\dot{h}) -\tfrac\kappa2 h^2\Big)^{q} \\
	+(2q)q\kappa\sC_h^{2q-1}\Big[
	\Big( \rho^2 (-2\ddot{h}+\sC_h(\dot{h})\dot{h}-\tfrac12\sC_h(\dot{h}^2)) -h \Big)
	 \Big( -\tfrac\kappa2 \rho^2\dot{h}^2 +\rho(R_h +\kappa h\dot{h}) -\tfrac\kappa2 h^2\Big)^{q-1}\Big].
\end{multline*}

Thus the coefficient of $d\rho\otimes d\rho$ in $\sum \alpha_q (\cR_g^{(2q)}-\lambda^{(2q)}g)$ is given by
\begin{equation*}
\begin{gathered}
\begin{multlined}
	\sum \alpha_q \frac{d\rho\otimes d\rho}{4\kappa\rho^2}\owedge \Big[
	\rho (q-1)\lrpar{-\frac\kappa2}^{q-1} \frac{(2q)!}{(n-2q+1)!}\frac{(n-2)!}2\sC_h^2(R_h+\kappa h\dot{h}) \\
	+ \rho^2 (q-1) \lrpar{-\frac\kappa2}^{q} \frac{(2q)!}{(n-2q+1)!}\frac{(n-2)!}2\sC_h^2(\dot{h}^2) \\
	-\rho^2\lrpar{-\frac\kappa2}^{q}\frac{(2q)!}{(n-2q+1)!}(n-1)! \sC_h(-2\ddot{h}+\sC_h(\dot{h})\dot{h}-\tfrac12\sC_h(\dot{h}^2)) \\
	+ \rho^2 \frac{(q-1)(q-2)}2\lrpar{-\frac\kappa2}^{q-2}
		\frac{(2q)!}{(n-2q+1)!}\frac{(n-4)!}{4!}\sC_h^4((R_h+\kappa h\dot{h})^2) 
	+\cO(\rho^3) \Big]	
\end{multlined} \\
\begin{multlined}
	=\frac{d\rho\otimes d\rho}{4\kappa\rho^2}\owedge \Big[ 
	\rho \bA_2(\alpha, \kappa)\sC_h^2(R_h+\kappa h\dot{h}) 
	+ \rho^2 \lrpar{-\frac\kappa2}(\bA_1(\alpha, \kappa)+3\bA_2(\alpha, \kappa))\sC_h^2(\dot{h}^2) \\
	+\rho^2\kappa(\bA_1(\alpha, \kappa)+2\bA_2(\alpha, \kappa))\sC_h(-2\ddot{h}+\sC_h(\dot{h})\dot{h}) \\
	+\rho^2\bA_4(\alpha, \kappa)\sC_h^4((R_h+\kappa h\dot{h})^2) 
	+\cO(\rho^3) \Big],
\end{multlined}
\end{gathered}
\end{equation*}
while the terms without $d\rho$ in $\sum \alpha_q (\cR_g^{(2q)}-\lambda^{(2q)}g)$ are given by
\begin{equation*}
\begin{gathered}
\begin{multlined}
	\rho^{-1} \sum \alpha_q \frac{(2q)!}{(n-2q+1)!}
	\Big[ \\
	\rho  \lrpar{-\frac\kappa2}^{q-1} \frac{(n-2)!}2
		((n-2q+1)\sC_h(R_h+\kappa h\dot{h}) + (q-1)h\sC_h^2(R_h+\kappa h\dot{h})) \\
	+\rho^2 \lrpar{-\frac\kappa2}^{q}  \frac{(n-2)!}2
		(2(n-2q+1)\sC_h(\dot{h}^2) + 3(q-1)\sC_h^2(\dot{h}^2)) \\
	-\rho^2\lrpar{-\frac\kappa2}^{q}\frac{(n-2)!}2 
		(2(n-2q+1)(-2\ddot{h}+\sC_h(\dot{h})\dot{h})
			+4(q-1)h\sC_h(-2\ddot{h}+\sC_h(\dot{h})\dot{h})) \\
	+\rho^2 \lrpar{-\frac\kappa2}^{q-2} \frac{(q-1)}2\frac{(n-4)!}{4!}
		(2(n-2q+1)\sC_h^3( (R_h+\kappa h\dot{h})^2) + (q-2)h\sC_h^4( (R_h+\kappa h \dot{h})^2 )) \\
	+\cO(\rho^3) \Big] 
\end{multlined}\\
\begin{multlined}
	= 
	\rho^{-1}
	\Big[
	\rho (\bA_1(\alpha, \kappa) \sC_h(R_h+\kappa h\dot{h}) + \bA_2(\alpha, \kappa)h\sC_h^2(R_h+\kappa h\dot{h})) \\
	+\rho^2 
	\lrpar{-\frac\kappa2}
		(2\bA_1(\alpha, \kappa)\sC_h(\dot{h}^2) + 3\bA_2(\alpha, \kappa)\sC_h^2(\dot{h}^2)) \\
	-\rho^2\lrpar{-\frac\kappa2}
		(2\bA_1(\alpha, \kappa)(-2\ddot{h}+\sC_h(\dot{h})\dot{h})
			+4\bA_2(\alpha, \kappa)h\sC_h(-2\ddot{h}+\sC_h(\dot{h})\dot{h})) \\
	+\rho^2 
		(\bA_3(\alpha, \kappa)\sC_h^3( (R_h+\kappa h\dot{h})^2) 
		+ \bA_4(\alpha, \kappa)h\sC_h^4( (R_h+\kappa h \dot{h})^2 )) 
	+\cO(\rho^3) \Big].
\end{multlined}
\end{gathered}
\end{equation*}
By contracting, we see that $\sum \beta_q(\scal^{(2q)}(g) - (n+1)\lambda^{(2q)})g$ is given by
\begin{multline*}
	\Big[
	\rho (\bA_1(\beta, \kappa) +(n+1)\bA_2(\beta, \kappa))\sC_h^2(R_h+\kappa h\dot{h}) \\
	+\rho^2 
	\lrpar{-\frac\kappa2}
		3(\bA_1(\beta, \kappa)+ (n+1)\bA_2(\beta, \kappa))\sC_h^2(\dot{h}^2) \\
	-\rho^2\lrpar{-\frac\kappa2}
		4(\bA_1(\beta, \kappa)+(n+1)\bA_2(\beta, \kappa))\sC_h(-2\ddot{h}+\sC_h(\dot{h})\dot{h}) \\
	+\rho^2 
		(\bA_3(\beta, \kappa)+(n+1)\bA_4(\beta, \kappa))\sC_h^4( (R_h+\kappa h\dot{h})^2) 
	+\cO(\rho^3) \Big]\lrpar{\frac{d\rho\otimes d\rho}{4\kappa\rho^2}+\frac{h}\rho}.
\end{multline*}

It follows that $F_g(\alpha, \beta)=0$ gives us from the $d\rho \otimes d\rho$ term the equation
\begin{multline}\label{eq:JN}
	d\rho\otimes d\rho\owedge \Big[ 
	\bB_{1,2}(\alpha, \beta, \kappa)\sC_h^2(R_h+\kappa h\dot{h}) \\
	+ \rho \lrpar{-\frac\kappa2}
	(\bA_1(\alpha, \kappa)+3\bB_{1,2}(\alpha, \beta, \kappa))\sC_h^2(\dot{h}^2) \\
	-\rho\lrpar{-\frac\kappa2}
	(2\bA_1(\alpha, \kappa)+4\bB_{1,2}(\alpha, \beta, \kappa))
		\sC_h(-2\ddot{h}+\sC_h(\dot{h})\dot{h}) \\
	+\rho
	\bB_{3,4}(\alpha, \beta, \kappa)
	\sC_h^4((R_h+\kappa h\dot{h})^2) 
	+\cO(\rho^2) \Big] =0
\end{multline} 
and from the term without $d\rho$ the equation
\begin{multline}\label{eq:JT}
	\bA_1(\alpha, \kappa) \sC_h(R_h+\kappa h\dot{h}) + 
	\bB_{1,2}(\alpha, \beta, \kappa)h\sC_h^2(R_h+\kappa h\dot{h}) \\
	+\rho
	\lrpar{-\frac\kappa2}
		(2\bA_1(\alpha, \kappa)\sC_h(\dot{h}^2) 
		+ 3\bB_{1,2}(\alpha, \beta, \kappa)h\sC_h^2(\dot{h}^2)) \\
	-\rho\lrpar{-\frac\kappa2}
		(2\bA_1(\alpha, \kappa)(-2\ddot{h}+\sC_h(\dot{h})\dot{h})
			+4\bB_{1,2}(\alpha, \beta, \kappa)h\sC_h(-2\ddot{h}+\sC_h(\dot{h})\dot{h})) \\
	+\rho 
		(\bA_3(\alpha, \kappa)\sC_h^3( (R_h+\kappa h\dot{h})^2) 
		+ \bB_{3,4}(\alpha, \beta, \kappa)
		h\sC_h^4( (R_h+\kappa h \dot{h})^2 )) 
	+\cO(\rho^2)=0.
\end{multline} 

Restricting \eqref{eq:JN} and \eqref{eq:JT} to $\rho=0$ yields the two equations
\begin{equation*}
\begin{gathered}
	\bB_{1,2}(\alpha, \beta, \kappa)\sC_h^2(R_h+\kappa h\dot{h})|_{\rho=0} =0 \\
	(\bA_1(\alpha, \kappa) \sC_h(R_h+\kappa h\dot{h}) + 
	\bB_{1,2}(\alpha, \beta, \kappa)h\sC_h^2(R_h+\kappa h\dot{h}))|_{\rho=0}=0
\end{gathered}
\end{equation*}
which, since $\bA_1(\alpha, \kappa)\neq 0,$ imply $\sC_h(R_h+\kappa h\dot{h})|_{\rho=0}=0.$
Contracting we have $\sC_h^2(R_h+\kappa h\dot{h})|_{\rho=0}=0,$ i.e.,
\begin{equation*}
	\sC_{h_0}(h_2) = -\frac1{2(n-1)\kappa}\sC_{h_0}^2(R_{h_0}),
\end{equation*}
 and hence
\begin{equation*}
\begin{gathered}
	\sC_{h_0}(R_{h_0}) + \kappa(h_0\sC_{h_0}(h_2)+(n-2)h_2) = 0\\
	\implies
	h_2 = -\frac1{(n-2)\kappa}\Big(\sC_{h_0}(R_{h_0})-\frac{\sC_{h_0}^2(R_{h_0})}{2(n-1)}h_0\Big)
	= -\frac1\kappa P(h_0)
\end{gathered}
\end{equation*}
where $P(h_0)$ denotes the Schouten tensor of $h_0.$\\

In particular it follows that
\begin{equation*}
	(R_h + \kappa h \dot{h})|_{\rho=0} = \mathrm{Weyl}_{h_0},
\end{equation*}
the Weyl curvature of $h_0,$ considered as a $(2,2)$-form.

To determine $h_4$ we differentiate \eqref{eq:JN} and \eqref{eq:JT} with respect to $\rho$ and set $\rho=0.$ First note that
\begin{equation*}
\begin{gathered}
\begin{multlined}
	\pa_{\rho}(\sC_h(R_h + \kappa h\dot{h}))|_{\rho=0} \\
	= \dot{\Ric}+ \kappa (h_2\sC_{h_0}(h_2) + h_0 (\tfrac12\sC_{h_0}^2(h_2^2)-\sC_{h_0}(h_2)^2 + \sC_{h_0}(h_4)) + (n-2)h_4), 
\end{multlined} \\
	\pa_{\rho}(\sC_h^2(R_h + \kappa h\dot{h}))|_{\rho=0}
	= \dot{\scal} + 2(n-1)\kappa
	(\tfrac12\sC_{h_0}^2(h_2^2)-\sC_{h_0}(h_2)^2 + \sC_{h_0}(h_4)),
\end{gathered}
\end{equation*}
where, using \cite[Theorem 1.174]{Besse}, cf. \cite[pg. 243]{Juhl:Book}, $\dot{\Ric}$ and $\dot{\scal}$ are given by
\begin{equation}\label{eq:DotRic}
\begin{gathered}
	\dot{\Ric} = \pa_t|_{t=0}\Ric(h_0+th_2)
	= \frac12\Delta_{L,h_0}(h_2) - \delta_{h_0}^*(\delta_{h_0}h_2) - \frac12\mathrm{Hess}_{h_0}(\sC_{h_0}(h_2)),\\
	\dot{\scal} = \pa_t|_{t=0}\scal(h_0+th_2)
	= \Delta_{h_0}(\sC_{h_0}(h_2))+\delta_{h_0}\delta_{h_0}h_2 - h_0(\Ric(h_0), h_2)\\
	= \sC_{h_0}(\dot{\Ric})
	+\tfrac12\sC_{h_0}^2(\Ric \owedge h_2)-\sC_{h_0}(h_2)\scal.
\end{gathered}
\end{equation}
Then, from \eqref{eq:JN} we have
\begin{multline}\label{eq:JNDer}
	2\kappa(-\bA_1(\alpha, \kappa)+(n-3)\bB_{1,2}(\alpha, \beta, \kappa))\sC_{h_0}(h_4) \\
	+ \lrpar{\frac\kappa2}
	(-\bA_1(\alpha, \kappa)+(2n-5)\bB_{1,2}(\alpha, \beta, \kappa)) \sC_{h_0}^2(h_2^2) \\
	-\kappa(-\bA_1(\alpha, \kappa)+2(n-2)\bB_{1,2}(\alpha, \beta, \kappa))\sC_{h_0}(h_2)^2 \\
	+
	(\bA_4(\alpha, \kappa) +\bA_3(\beta, \kappa) +(n+1)\bA_4(\beta, \kappa))
	\sC_{h_0}^4(\Weyl_{h_0}^2) 
	+\bB_{1,2}(\alpha, \beta, \kappa) \dot{\scal}
	 =0.
\end{multline} 
and from \eqref{eq:JT} we have 
\begin{multline}\label{eq:JTDer}
	(n-4)\kappa \bA_1(\alpha, \kappa) h_4 
	+ \kappa (\bA_1(\alpha, \kappa) + 2(n-3)\bB_{1,2}(\alpha, \beta, \kappa))h_0 \sC_{h_0}(h_4) \\
	-\kappa \bA_1(\alpha, \kappa) \sC_{h_0}(h_2^2) 
	+ \frac\kappa2(\bA_1(\alpha, \kappa)+(2n-5)\bB_{1,2}(\alpha, \beta, \kappa))h_0 \sC_{h_0}^2(h_2^2)\\
	+2\kappa\bA_1(\alpha, \kappa) h_2\sC_{h_0}(h_2) 
	- \kappa(\bA_1(\alpha, \kappa) +2(n-2)\bB_{1,2}(\alpha, \beta, \kappa))h_0\sC_{h_0}(h_2)^2\\
	+ \bA_3(\alpha, \kappa) \sC_{h_0}^3( \Weyl_{h_0}^2) 
		+ \bB_{3,4}(\alpha, \beta, \kappa)
		h_0\sC_{h_0}^4( \Weyl_{h_0}^2 ) 
	+\bA_1(\alpha, \kappa) \dot{\Ric} + \bB_{1,2}(\alpha, \beta, \kappa) h_0 \dot{\scal} =0.
\end{multline} 
The contraction of the latter is
\begin{multline}\label{eq:JTDerC}
	\kappa (2(n-2)\bA_1(\alpha, \kappa) +  2n(n-3)\bB_{1,2}(\alpha, \beta, \kappa)) \sC_{h_0}(h_4) \\
	+ \frac\kappa2((n-2)\bA_1(\alpha, \kappa)+n(2n-5)\bB_{1,2}(\alpha, \beta, \kappa)) \sC_{h_0}^2(h_2^2)\\
	- \kappa((n-2)\bA_1(\alpha, \kappa) +2n(n-2)\bB_{1,2}(\alpha, \beta, \kappa))\sC_{h_0}(h_2)^2\\
	+ (\bA_3(\alpha, \kappa) + n\bB_{3,4}(\alpha, \beta, \kappa))\sC_{h_0}^4( \Weyl_{h_0}^2 ) 
	+\bA_1(\alpha, \kappa) \sC_h(\dot{\Ric}) + n\bB_{1,2}(\alpha, \beta, \kappa) \dot{\scal} =0.
\end{multline} 

Multiplying \eqref{eq:JNDer} by $n$ and subtracting it from \eqref{eq:JTDerC} yields
\begin{multline*}
	\kappa (4(n-1)\bA_1(\alpha, \kappa) ) \sC_{h_0}(h_4) 
	+ \lrpar{\frac\kappa2}
	(2(n-1)\bA_1(\alpha, \kappa)) \sC_{h_0}^2(h_2^2) 
	- \kappa(2(n-1)\bA_1(\alpha, \kappa) )\sC_{h_0}(h_2)^2 \\
	+\bA_3(\alpha, \kappa) \sC_{h_0}^4( \Weyl_{h_0}^2 )  
	+\bA_1(\alpha, \kappa) \sC_h(\dot{\Ric})  
	=0
\end{multline*} 
and so we find
\begin{equation*}
	\sC_{h_0}(h_4) =
	-\frac14\sC_{h_0}^2(h_2^2) + \frac12\sC_{h_0}(h_2)^2 
	-\frac1{4\kappa(n-1)}\sC_h(\dot{\Ric})
	-\frac{\bA_3(\alpha, \kappa)}{4\kappa(n-1)\bA_1(\alpha, \kappa)}\sC_{h_0}^4(\Weyl_{h_0}^2).
\end{equation*}
Substituting into \eqref{eq:JTDer} we find
\begin{multline}\label{eq:h4}
	h_4 
	= 
	-\frac1{(n-4)}
	\Big( 
	-\frac{h_0\sC_h(\dot{\Ric})}{4\kappa(n-1)}
	+h_0(\tfrac14 \sC_{h_0}^2(h_2^2)
	- \tfrac12 \sC_{h_0}(h_2)^2)
	- \sC_{h_0}(h_2^2) 
	+2 h_2\sC_{h_0}(h_2) 
	+\frac{\dot{\Ric}}{\kappa}
	\Big) \\
	-\frac{\bB_{1,2}(\alpha, \beta, \kappa)h_0}{(n-4) \bA_1(\alpha, \kappa)}	
	\Big(
	\tfrac12 (2n-5) \sC_{h_0}^2(h_2^2)
	- 2(n-2)\sC_{h_0}(h_2)^2 
	+ \frac{\dot{\scal}}{\kappa} \Big)\\
	+2(n-3)
	\Big(
	-\frac14\sC_{h_0}^2(h_2^2) + \frac12\sC_{h_0}(h_2)^2 
	-\frac1{4\kappa(n-1)}\sC_h(\dot{\Ric})
	-\frac{\bA_3(\alpha, \kappa)}{4\kappa(n-1)\bA_1(\alpha, \kappa)}\sC_{h_0}^4(\Weyl_{h_0}^2)
	 \Big) \\
	 \\
	-\frac{
	4(n-1)\bA_3(\alpha, \kappa) \sC_{h_0}^3( \Weyl_{h_0}^2) 
	+ (4(n-1) \bB_{3,4}(\alpha, \beta, \kappa)-\bA_3(\alpha, \kappa))
		h_0\sC_{h_0}^4( \Weyl_{h_0}^2 ) 
	}{4\kappa(n-1)(n-4)\bA_1(\alpha, \kappa)} 
\end{multline} 
%

\section{Graham-Lee existence} \label{sec:GL}

In this section we use the results of Graham-Lee \cite{Graham-Lee} to show that there are many Poincar\'e-Lovelock metrics on the interior of the Euclidean ball. 

\begin{theorem}\label{thm:GLA}
Let $M = \bbB^{n+1},$ $n\geq 4,$ ${\df h}$ the hyperbolic metric on $M$ and $\hat {\df h} = \rho^2{\df h}$ the round metric on $\bbS^n = \pa M.$
For any smooth Riemannian metric $\hat g$ on $\bbS^{n}$ which, for some $\theta>0,$ is sufficiently close in $\cC^{2,\theta}(M, \cS^2(M))$ to $\hat {\df h}$ there is a metric $g \in \cC^{\infty}(M, \cS^2(M))\cap \rho^{-2}\cC^{0}(\bar M, \cS^2(M))$ satisfying 
\begin{equation}\label{eq:GLeq1}
	\begin{cases}
	\cR^{(2q)}_g - \lambda^{(2q)}g = 0\\
	x^2 g\rest{\pa M} \text{ is conformal to } \bar g,
	\end{cases}
\end{equation}
and, for any $\alpha$ such that $\mathrm{LimSec}(\alpha) \neq \emptyset,$ there is a metric 
$g \in \cC^{\infty}(M, \cS^2(M))\cap \rho^{-2}\cC^{0}(\bar M, \cS^2(M))$ satisfying
\begin{equation}\label{eq:GLeq2}
	\begin{cases}
	\sum \alpha_q (\cE^{(2q)}_g - (1-\tfrac{n+1}{2q})\lambda^{(2q)}g) =0 \\
	x^2 g\rest{\pa M} \text{ is conformal to } \bar g.
	\end{cases}
\end{equation}
\end{theorem}

(In \cite{Graham-Lee}, the equation \eqref{eq:GLeq1} is treated with $q=1.$ Solving \eqref{eq:GLeq2} with $\alpha = e_q$ (i.e., $\alpha_j = \delta_{jq}$) gives a solution to \eqref{eq:GLeq1}. We treat both equations in parallel as it makes it simpler to compare with  \cite{Graham-Lee}.)

To compensate for diffeomorphism invariance, we will study a perturbation of the equation of the previous section,
\begin{equation*}
	Q_{(\alpha,\beta)}(g,t) 
	= \sum \alpha_q (\cR^{(2q)}_g - \lambda^{(2q)}g) + \beta_q (\ell^{(2q)}_g - (n+1)\lambda^{(2q)})g -\Phi_{(\alpha,\beta)}(g,t)=0,
\end{equation*}
where $t$ is an auxiliary metric and $\Phi_{\alpha, \beta}(g,t)$ an operator specified below \eqref{eq:DefPhi}.
We will show that the linearization of $Q_{(\alpha,\beta)}(g,t)$ is asymptotically equal to a linear combination of $(\Delta_g + 2n)$ in pure-trace directions and $(\Delta_g-2)$ in trace-free directions.

Once the linearizations are computed, the arguments in \cite{Graham-Lee} will apply virtually unchanged.
Given a metric $\hat g$ on $\bbS^{n+1},$ we define an asymptotically hyperbolic metric $T(\hat g)$ extending the conformal class of $\hat g$ into $M = \bbB^{n+1}$ and equal to ${\df h}$ away from $\pa \bbB^{n+1}$ in \eqref{eq:DefT} below.
Using the linearization of $Q_{(\alpha,\beta)}(g,t),$ and an analysis of the corresponding `indicial operators', Graham-Lee constructed an operator $S(\hat g)$ depending smoothly on $\hat g$ such that, e.g.,
\begin{equation*}
	Q_{(\alpha,\beta)}(S(\hat g), T(\hat g)) = \cO(\rho^{n-2})
\end{equation*}
essentially by showing that the construction of the asymptotic expansion in the previous section can be carried out smoothly.
Using these approximate solutions, the arguments in \cite{Graham-Lee} show that as long as $\hat g$ is sufficiently close to the round metric and
\begin{equation*}
	\sum \lambda^{(2q)}
	( q(\alpha_q+(n+1)\beta_q) ) \neq 0, \quad
	\bA_1(\alpha, \kappa)\neq 0,
\end{equation*}
there is a metric $g$ extending the conformal class of $\hat g,$ such that
\begin{equation*}
	Q_{(\alpha,\beta)}(g, T(\hat g)) = 0.
\end{equation*}
Finally, in Lemma \ref{lem:GL22} below we show that 
if $(\alpha, \beta)$ is given by $(e_q,0)$ or if $(\alpha,\beta)$ satisfies $\beta_q = -\tfrac1{2q}\alpha_q,$ then 
\begin{equation*}
	Q_{(\alpha,\beta)}(g, T(\hat g)) =0 \implies \Phi_{(\alpha,\beta)}(g, T(\hat g))=0.
\end{equation*}
$ $\\

We start by computing the linearizations of the generalized Ricci tensors and Lovelock scalars.
These are due to \cite{DeLimaSantos:Defs} at constant curvature metrics and \cite{CaulaDeLimaSantos} for slightly more general metrics.

Let us introduce the following notation, with $g$ and $t$ two metrics on $M,$
\begin{equation*}
\begin{aligned}
	\delta_g:& \text{ symmetric 2-tensors } \lra \text{ one-forms }, \quad
	(\delta_gt)_i = -g^{jk}t_{ij,k}, \\
	\delta_g:& \text{ one-forms } \lra \text{ functions }, \quad
	\delta_g\omega = -g^{jk}\omega_{j,k}, \\
	\delta_g^*:& \text{ one-forms } \lra \text{ symmetric 2-tensors }, \quad
	(\delta_g^*\omega)_{ij} = \tfrac12(\omega_{i,j} + \omega_{j,i}), \\
	gt^{-1}:& \text{ one-forms } \lra \text{ one-forms }, \quad
	(gt^{-1}\omega)_i = g_{ij}(t^{-1})^{jk}\omega_k, \\
	\cG_g^{(2q)}:& \text{ symmetric 2-tensors } \lra \text{ symmetric 2-tensors }, \quad
	\cG_g^{(2q)}(\Phi)_{ij} = \Phi_{ij} - \tfrac1{2q}g^{k\ell}\Phi_{k\ell}g_{ij}.
\end{aligned}
\end{equation*}
The latter, known as the  $2q$-gravitation operator, has the key property that it takes $\cR^{(2q)}_g$ to $\cE^{(2q)}_g,$ the $(2q)^{\text{th}}$ Einstein-Lovelock tensor, and hence the second Bianchi identities read
\begin{equation*}
	\delta_g\cG_g^{(2q)}(\cR^{(2q)}) = 0.
\end{equation*}
We point out that, if $\omega$ is a one-form, then
\begin{equation*}
	(\cG_g^{(2)}\delta_g^*\omega)_{ij} = \tfrac12(\omega_{i,j}+\omega_{j,i}-g^{st}\omega_{s,t}g_{ij})
	=\delta_g^*\omega+\tfrac12g\delta_g\omega
\end{equation*}

\begin{lemma}\label{lem:LinCurv}
Let $(M, g_0)$ be an asymptotically hyperbolic manifold, $\rho$ a boundary defining function that is special for $g_0,$ and let $r$ be a symmetric two tensor of the form $r = \rho^N \bar r$ with $\bar r \in \cC^2(\bar M, \cS^2(M)).$

The linearization of the map $g \mapsto \cR^{(2q)}_g-\lambda^{(2q)}g$ at $g_0$ in the direction of $r$ satisfies
\begin{multline}\label{eq:LinL}
	\cD \lrpar{ \cR^{(2q)}_g -\lambda^{(2q)}g}_{g_0}(r) \\
	= \frac{\lambda^{(2q)}}{n(n-1)}
	\Big( 
	(
	(2q-nq-1)\sC_{g_0}(r) 
	-(q-1)
	( \tfrac12\Delta_{g_0}(\sC_{g_0}(r)) +\delta_{g_0}\delta_{g_0} \cG_{g_0}^{(2)}(r) )){g_0} \\
	- (n-2q+1)
	(\tfrac12\Delta_{g_0}(r)-\delta_{g_0}^*\delta_{g_0}\cG_{g_0}^{(2)}(r)-r)
	\Big) + \cO(\rho^{N+1}).
\end{multline}
If $g_0$ has constant sectional curvature then the $\cO(\rho^{N+1})$ term is identically zero.

The linearization of the map $g \mapsto (\scal^{(2q)}(g)-(n+1)\lambda^{(2q)})g$ at $g_0$ in the direction of $r$ satisfies
\begin{multline}\label{eq:Linell}
	\cD \lrpar{ (\scal^{(2q)}(g)-(n+1)\lambda^{(2q)})g }_{g_0}(r) \\
	= -\frac{q\lambda^{(2q)}}n
	\Big(n\sC_{g_0}(r) 
	+ \tfrac12\Delta_{g_0}(\sC_{g_0}(r)) +\delta_{g_0}\delta_{g_0} \cG_{g_0}^{(2)}(r) 
	 \Big)g_0 + \cO(\rho^{N+1}).
\end{multline}
If $g_0$ has constant sectional curvature then the $\cO(\rho^{N+1})$ term is identically zero.
\end{lemma}

In particular, this implies that
the linearization of $g \mapsto \cE^{(2q)}_g - (1-\tfrac{n+1}{2q})\lambda^{(2q)}g$ at $g_0$ in the direction of $r$ satisfies 
\begin{multline}\label{eq:LinE}
	\cD(\cE^{(2q)}_g - (1-\tfrac{n+1}{2q})\lambda^{(2q)}g)_{g_0}(r) \\
	= \frac{(n-2q+1)\lambda^{(2q)}}{n(n-1)}
	\Big( 
	\tfrac12{g_0}
	(\tfrac12\Delta_{g_0}(\sC_{g_0}(r)) +\delta_{g_0}\delta_{g_0} \cG_{g_0}^{(2)}(r) +(n-2)\sC_{g_0}(r)) \\
	- (\tfrac12\Delta_{g_0}r-\delta_{g_0}^*\delta_{g_0}\cG_{g_0}^{(2)}(r)-r)
	 \Big) + \cO(\rho^{N+1}),
\end{multline}
where, 
if $g_0$ has constant sectional curvature, then the $\cO(\rho^{N+1})$ term is identically zero.

\begin{proof}
Let $g=g(s)$ be a family of metrics on $M$ with $g(0) = g_0,$ $g'(0) = r.$

We write the linearization of $\cR^{(2q)}_g$ at $g_0$ as a sum of two operators according to
\begin{equation*}
\begin{multlined}
	\cD\lrpar{ \cR^{(2q)}_g }_{g_0}(r)
	= \frac{\pa}{\pa s}\rest{s=0}\lrpar{ \sC_g^{2q-1}(R^{q}) } \\
	= \frac{\pa}{\pa s}\rest{s=0}\lrpar{ \sum g^{a_1b_1}g^{a_2b_2}\cdots g^{a_{2q-1}b_{2q-1}} 
	R_{c_1d_1e_1f_1}(g)\cdots R_{c_{q}d_qe_qf_q}(g) } \\
	= \sum \sum {g_0}^{a_1b_1} \cdots \frac{\pa}{\pa s}\rest{s=0}(g^{a_kb_k}) \cdots {g_0}^{a_{2q-1}b_{2q-1}} 
	R_{c_1d_1e_1f_1}(g_0)\cdots R_{c_{q}d_qe_qf_q}(g_0) \\
	+ \sum \sum {g_0}^{a_1b_1} \cdots {g_0}^{a_{2q-1}b_{2q-1}} 
	R_{c_1d_1e_1f_1}(g_0)\cdots \frac{\pa}{\pa s}\rest{s=0}(R_{c_qd_qe_qf_q}(g)) \cdots R_{c_qd_qe_qf_q}(g_0) \\
	= \cL^{(2q)}(r) + \cM^{(2q)}(r).
\end{multlined}
\end{equation*}
For $\cL^{(2q)}(r)$ note that the factors $g_0^{a_ib_i}$ are $\cO(\rho^2),$ 
$\frac{\pa}{\pa s}\rest{s=0}(g^{a_kb_k}) = -g_0^{a_k\wt a_k}r_{\wt a_k \wt b_k} g_0^{\wt b^kb^k}$ is $\cO(\rho^{4+N})$ and each factor of $R_{c_id_ie_if_i}(g_0)$ is equal to 
$(-\tfrac12 g_0^2)_{c_id_ie_if_i} + \cO(\rho^{-3}).$
Hence
\begin{multline*}
	\cL^{(2q)}(r) \\
	= \sum \sum g_0^{a_1b_1} \cdots \frac{\pa}{\pa s}\rest{s=0}(g^{a_kb_k}) \cdots g_0^{a_{2q-1}b_{2q-1}} 
	(-\tfrac12 g_0^2)_{c_1d_1e_1f_1}\cdots (-\tfrac12 g_0^2)_{c_{q}d_qe_qf_q}  + \cO(\rho^{N+1}) \\
	= \cL^{(2q)}_0(r) + \cO(\rho^{N+1})
\end{multline*}
and if $g_0$ has constant sectional curvature then $\cL^{(2q)}(r) = \cL^{(2q)}_0(r).$
We can use Lemma \ref{lem:DoubleForms} to see that $\cL^{(2q)}_0(r)$ satisfies
\begin{equation*}
\begin{gathered}
	\frac{\pa}{\pa s}\rest{s=0}(\lrpar{ \sC^{2q-1}_{g}((-\tfrac12{g}^2)^q) }
	= \cL_{0}^{(2q)}(r) + (2q)(-\tfrac12)^q \sC^{2q-1}_{g_0}( g_0^{2q-1} r) \\
	\implies
	\cL_{0}^{(2q)}(r) 
	= \frac{(-1)^q}{2^q}\frac{(2q)!(m-2)!}{(m-2q)!} (2q-1) (r-{g_0}\sC_{g_0}(r)).
\end{gathered}
\end{equation*}

Similarly, in the expression for $\cM^{(2q)}(r)$ note that $\dot{R_g} = \pa_s|_{s=0}(R_g)$ is $\cO(\rho^{t-2})$ (e.g., from \cite[Theorem 1.174(c)]{Besse}), hence 
\begin{multline*}
	\cM^{(2q)}(r) 
	= \sum \sum g_0^{a_1b_1} \cdots g_0^{a_{2q-1}b_{2q-1}} \\
	(-\tfrac12 g_0^2)_{c_1d_1e_1f_1}\cdots 
	\frac{\pa}{\pa s}\rest{s=0}(R_{c_qd_qe_qf_q}(g)) \cdots 
	(-\tfrac12 g_0^2)_{c_qd_qe_qf_q} + \cO(\rho^{N+1}) \\
	= \cM^{(2q)}_0(r) + \cO(\rho^{N+1})
\end{multline*}
and if $g_0$ has constant sectional curvature then $\cM^{(2q)}(r) = \cM^{(2q)}_0(r).$
We can compute $\cM^{(2q)}_0(r)$ as
\begin{multline*}
	\cM_{0}^{(2q)}(r) 
	=q\sC_{g_0}^{2q-1}((-\tfrac12g_0^2)^{q-1}\dot{R_g}) \\
	=  -\frac{(-1)^{q}}{2^{q}} \frac{(2q)!(m-3)!}{(m-2q)!} \lrpar{
	(q-1)g_0\sC_{g_0}^2(\dot{R_g}) + (m-2q)\sC_{g_0}(\dot{R_g}) }
\end{multline*}
and hence
\begin{equation}\label{eq:LinLpre}
\begin{gathered}
\begin{multlined}
	\cD\lrpar{ \cR^{(2q)}_g }_{g_0}(r)
	= \frac{(-1)^{q}}{2^{q}} \frac{(2q)!(m-3)!}{(m-2q)!} \Big( (m-2)(2q-1)(r-{g_0}\sC_{g_0}(r)) \\
	-(q-1){g_0}\sC_{g_0}^2(\dot{R_g}) - (m-2q)\sC_{g_0}(\dot{R_g}) \Big) + \cO(\rho^{N+1}) 	
\end{multlined} \\
\begin{multlined}
	= \frac{\lambda^{(2q)}}{n(n-1)}
	 \Big( (n-1)(2q-1)(r-{g_0}\sC_{g_0}(r)) \\
	-(q-1){g_0}\sC_{g_0}^2(\dot{R_g}) - (n-2q+1)\sC_{g_0}(\dot{R_g}) \Big) + \cO(\rho^{N+1}) 	
\end{multlined} 
\end{gathered}
\end{equation}
The variation of the curvature tensor has contractions \cite[Theorem 1.174]{Besse}, \cite[(3.7), (3.8)]{DeLimaSantos:Defs}
\begin{equation*}
\begin{gathered}
	\sC_{g_0}(\dot{R_g}) = \tfrac12\Delta_{g_0}r-\delta_{g_0}^*\delta_{g_0}\cG_{g_0}^{(2)}(r)
		+\tfrac12( \cR_g^{(2)}\circ r + r \circ \cR_g^{(2)} ) \\
	\sC_{g_0}^2(\dot{R_g}) = \tfrac12\Delta_{g_0}(\sC_{g_0}(r)) +\delta_{g_0}\delta_{g_0} \cG_{g_0}^{(2)}(r) + g_0(\cR_{g_0}^{(2)}, r)
\end{gathered}
\end{equation*}
and we note that 
\begin{equation*}
\begin{gathered}
	\tfrac12(\cR_g^{(2)}\circ r + r \circ \cR_g^{(2)})_{ij} 
	=-(m-1)r + \cO(\rho^{N+1}) \\
	g_0(\cR_{g_0}^{(2)}, r)
	= -(m-1)\sC_{g_0}(r) + \cO(\rho^{N+1})
\end{gathered}
\end{equation*}
with the $\cO(\rho^{N+1})$ terms vanishing if $g_0$ has constant sectional curvature.
Substituting these expressions into \eqref{eq:LinLpre} yields \eqref{eq:LinL}.\\

Similarly decomposing 
$\cD(\scal^{(2q)}(g))_{g_0}(r) = \bar\cL_0^{(2q)}(r) + \bar\cM_0^{(2q)}(r) + \cO(\rho^{N+1})$ we find
\begin{equation*}
\begin{gathered}
	\frac{\pa}{\pa s}\rest{s=0}\lrpar{ \sC^{2q}_{g}((-\tfrac12{g}^2)^q) }
	= \bar\cL_{0}^{(2q)}(r) + (2q)(-\tfrac12)^q \sC^{2q}_{g_0}( g_0^{2q-1} r) \\
	\implies
	\bar\cL_{0}^{(2q)}(r) 
	= -(2q)\frac{(-1)^q}{2^q}\frac{(2q)!(m-1)!}{(m-2q)!} \sC_{g_0}(r)
\end{gathered}
\end{equation*}
and
\begin{equation*}
	\bar\cM_{0}^{(2q)}(r) 
	=q\sC_{g_0}^{2q}((-\tfrac12{g_0}^2)^{q-1}\dot{R_g}) \\
	= q \lrpar{-\frac12}^{q-1}\frac{(2q)!}{(m-2q)!}\frac{(m-2)!}2 \sC_{g_0}^2(\dot{R_g})
\end{equation*}
so that
\begin{equation*}
\begin{gathered}
	\cD(\scal^{(2q)}(g))_{g_0}(r) 
	= q \lrpar{-\frac12}^{q-1}\frac{(2q)!}{(m-2q)!}\frac{(m-2)!}2 
	\Big(2(m-1)\sC_{g_0}(r) + \sC_{g_0}^2(\dot{R_g}) \Big) + \cO(\rho^{N+1}) \\
	= -\frac{q}n  \lambda^{(2q)}
	\Big(n\sC_{g_0}(r) 
	+ \tfrac12\Delta_{g_0}(\sC_{g_0}(r)) +\delta_{g_0}\delta_{g_0} \cG_{g_0}^{(2)}(r)
	 \Big) + \cO(\rho^{N+1}).
\end{gathered}
\end{equation*}
\end{proof}

To compensate for the diffeomorphism invariance of these tensors, we will perturb them by adding an operator of the form
\begin{equation*}
	(c_1\delta_g^*+c_2g\delta_g)(gt^{-1}\delta_g\cG^{(2)}_g(t))
\end{equation*}
where $t$ is an auxiliary metric.

\begin{lemma}[{\cite[Lemma 2.3, Proposition 2.10]{Graham-Lee}}] \label{lem:LinPert}
For metrics $(g,t),$ a symmetric $2$-tensor $r,$ and constants $c_1,$ $c_2,$ the linearization of the map 
$(g,t) \mapsto (c_1\delta_g^*+c_2g\delta_g)(gt^{-1}\delta_g\cG^{(2)}_g(t))$ with respect to the first variable in the direction of $r$ is
\begin{equation*}
	D_1((c_1\delta_g^*+c_2g\delta_g)(gt^{-1}\delta_g\cG^{(2)}_g(t)))_{(g,t)}(r)
	= (c_1\delta_g^*+c_2g\delta_g)(-\delta_g\cG^{(2)}_g(r) + \sC(r)-\sD(r))
	+ \sB_{c_1,c_2}(r)
\end{equation*}
where, with covariant derivatives with respect to $g,$
\begin{equation*}
\begin{gathered}
	C_{ij}^k = \tfrac12(t^{-1})^{k\ell}(t_{i\ell,j}+t_{j\ell,i}-t_{ij,\ell}), \quad
	D^k = g^{ij}C_{ij}^k = -(t^{-1}\delta_g\cG^{(2)}_gt)^k, \\
\begin{multlined}
	(\sB_{c_1,c_2}(r))_{jk}
	=c_1 \tfrac12D^{t}(r_{kt,j}+r_{jt,k}-r_{kj,t}) \\
	-c_2 \Big( r^{ab}D_{a,b}g_{jk} 
	+\tfrac12D^{t}g^{ab}(r_{at,b}+r_{bt,a}-r_{ab,t})
	- g^{ab}D_{a,b}r_{jk} \Big), 
\end{multlined}\\
	(\sC(r))_j = g_{jk}C^k_{ab}r^{ab}, \quad
	(\sD(r))_j = D^kr_{jk}.
\end{gathered}
\end{equation*}
If $g$ and $t$ are asymptotically hyperbolic metrics such that $\rho^2g|_{\pa M}=\rho^2t|_{\pa M}$ and $r = \rho^N \bar r,$ with $\bar r\in \cC^2(\bar M, \cS^2(M)),$ then
\begin{equation*}
	D_1((c_1\delta_g^*+c_2g\delta_g)(gt^{-1}\delta_g\cG^{(2)}_g(t)))_{(g,t)}(r)	
	= (c_1\delta_g^*+c_2g\delta_g)(-\delta_g\cG^{(2)}_g(r)) + \cO(\rho^{N+1}),
\end{equation*}
and if moreover $g$ and $t$ are equal on $M$ 
then the $\cO(\rho^{N+1})$ term vanishes.
\end{lemma}
\begin{proof}
For $c_1=1,$ $c_2=0,$ this is shown in \cite{Graham-Lee}.
So it suffices to compute
\begin{multline*}
	\pa_s|_{s=0}
	((g_s\delta_{g_s})(gt^{-1}\delta_g\cG^{(2)}_g(t)))
	= 
	-\pa_s|_{s=0}
	((g_s\delta_{g_s})gD) \\
	= - r^{ab}D_{a,b}g_{jk} 
	+\tfrac12D^{t}g^{ab}(r_{at,b}+r_{bt,a}-r_{ab,t})
	- g^{ab}D_{a,b}r_{jk}
\end{multline*}
and note that if $g$ and $t$ are asymptotically hyperbolic metrics with the same leading term at $\pa M$ and $r$ is as above, then $D^k = \cO(\rho^2),$ $D^k_{\phantom{k},s} = \cO(\rho),$ and hence the right hand side of this expression is $\cO(\rho^{N+1}).$ It  vanishes if  $g=t$ since then both $C$ and $D$ vanish.
\end{proof}

In view of Lemmas \ref{lem:LinCurv} and \ref{lem:LinPert}, we define
\begin{multline}\label{eq:DefPhi}
	\Phi_{(\alpha,\beta)}(g,t) \\
	= -
	\sum\frac{\lambda^{(2q)}}{n(n-1)}
	\Big( \alpha_q (n-2q+1) \delta_g^* - (\alpha_q(q-1)+\beta_q(n-1)q)g\delta_g \Big)
	(gt^{-1}\delta_g\cG^{(2)}_g(t)),
\end{multline}
and, as anticipated above,
\begin{equation*}
	Q_{(\alpha,\beta)}(g,t) 
	= \sum \alpha_q (\cR^{(2q)}_g - \lambda^{(2q)}g) + \beta_q (\scal^{(2q)}(g) - (n+1)\lambda^{(2q)})g -\Phi_{(\alpha,\beta)}(g,t).\end{equation*}
We have shown the following.

\begin{lemma}\label{lem:Linear}
Assume $g_0$ and $t$ are asymptotically hyperbolic metrics such that $\rho^2g_0|_{\pa M}=\rho^2t|_{\pa M}$ and $r = \rho^N \bar r,$ where $\bar r\in \cC^2(\bar M, \cS^2(M))$ has decomposition 
\begin{equation*}
	r= u g_0 + r_0 \Mwith \sC_{g_0}(r_0)=0.
\end{equation*}
The linearization of $g \mapsto Q_{(\alpha,\beta)}(g,t)$ at $(g_0,g_0)$ in the direction of $r$ is 
\begin{multline}\label{eq:LinPureL}
	-\sum\frac{\lambda^{(2q)}}{2n(n-1)} 
	\Big( q(n-1)(\alpha_q+(n+1)\beta_q) (\Delta_g+2n)(ug_0)
	+ (n-2q+1)\alpha_q(\Delta_g-2)(r_0) \Big) \\
	+ \cO(\rho^{N+1})
\end{multline}
If $g_0$ has constant sectional curvature then the $\cO(\rho^{N+1})$ term is identically zero.
\end{lemma}

In particular if $\beta_q = -\tfrac1{2q}\alpha_q$ the linearization is
\begin{equation}\label{eq:LinMixedL}
	\frac{\bA_1(\alpha, \kappa)}4
	\Big( -(n-1)(\Delta_{g_0}+2n)(u{g_0})+
	2(\Delta_{g_0}-2)(r_0)
	\Big).
\end{equation}

Recall the notation:
If $\bar g_i$ are metrics, all assumed to be of class $\cC^k$ on $M,$ then 
\begin{equation*}
	\sE^k(\bar g_1, \ldots, \bar g_N)
\end{equation*}
will denote any tensor whose components in any coordinate system smooth up to $\pa M$ are polynomials, with coefficients in $\CI(\bar M)$ in the components of the $\bar g_i,$ $\bar g_i^{-1},$ and their partial derivatives, such that in each term the total number of derivatives of the $\bar g_i$ that appear is at most $k.$

If $g$ is conformally compact, then
\begin{equation*}
	R_{jk} = -\rho^2(n\bar g^{it}\rho_i\rho_t)\bar g_{jk} + \rho^{-1}\sE^1(\bar g) + \sE^2(\bar g)
\end{equation*}
and more generally, from \eqref{eq:R2q},
\begin{equation*}
	\sC_g^{2q-1}(R^q) = \rho^{-2}\sum_{j=1}^{2q} \rho^j\sE^j(\bar g), \Mand
	\sC_g^{2q}(R^q) = \sum_{j=1}^{2q} \rho^j\sE^j(\bar g)
\end{equation*}

\begin{lemma}
For $g$ and $t$ conformally compact metrics,
\begin{multline}\label{eq:LeadTerm}
	(c_1\delta_g^*+c_2g\delta_g)(gt^{-1}\delta_g\cG^{(2)}_g(t)) \\
	= \rho^{-2}(\tfrac12 c_1(B_k\rho_j + B_j \rho_k) + c_2 g^{st}B_s\rho_t g_{jk}) + \rho^{-1}\sE^1(\bar g, \bar t) + \sE^2(\bar g, \bar t)
\end{multline}
where $B = [ \sC_{\bar g}(\bar t) \bar g\bar t^{-1} - (n+1) ] \; d\rho.$
In particular, 
$(c_1\delta_g^*+c_2g\delta_g)(gt^{-1}\delta_g\cG^{(2)}_g(t)) = \cO(\rho^{-1})$
if $\bar g\rest{\pa M}=\bar t|_{\pa M}.$

\end{lemma}

\begin{proof}
Graham-Lee \cite[Proof of Proposition 2.5]{Graham-Lee} compute that 
\begin{equation*}
	(gt^{-1}\delta_g\cG^{(2)}_gt)_k = -\rho^{-1}B_k + \sE^1(\bar g, \bar t).
\end{equation*}
Applying $(c_1\delta_g^*+c_2g\delta_g)$ to this expression yields \eqref{eq:LeadTerm}.
If $\bar g\rest{\pa M}=\bar t|_{\pa M}$ then $B=\cO(\rho).$
\end{proof}

Recall, from \cite[\S 3]{Graham-Lee} the following notation for spaces of functions.
Consider 
$\hat \Omega$ a bounded open subset of $\bbR^{n+1}$ with smooth boundary and 
$\Omega$ an open subset of $M.$ 
Let $d_x$ denote the Euclidean distance from $x$ to $\pa \hat \Omega$ and denote for $s \in \bbR,$ $k \in \bbN,$ $\gamma\in(0,1),$
\begin{equation*}
\begin{gathered}
	\norm{u}^{(s)}_{k,0;\Omega}
	= \sum_{j=0}^k\sum_{|\xi|=j} \norm{d^{-s+j}D^{\xi}u}_{L^{\infty}(\Omega)}\\
	\norm{u}^{(s)}_{k,\gamma;\Omega}
	= \norm{u}^{(s)}_{k,0;\Omega}
	+ \sum_{|\xi|=k}
	\Big[ \min(d_x^{-s+k+\gamma}, d_y^{-s+k+\gamma})\frac{|\pa^{\xi}u(x)-\pa^{\xi}u(y)|}{|x-y|^{\alpha}}\Big].
\end{gathered}
\end{equation*}
We denote by $\Lambda^s_{k,0}(\Omega),$ $\Lambda^s_{k,\gamma}(\Omega),$  the Banach spaces of functions in $\cC^k(\Omega),$ with finite $\norm{\cdot}^{(s)}_{k,0;\Omega}$ or finite $\norm{\cdot}^{(s)}_{k,\gamma;\Omega},$ respectively.
These give rise to Banach spaces of functions on $M,$ denoted $\Lambda^s_{k,0}(M),$ $\Lambda^s_{k,\gamma}(M),$ see \cite[Proposition 3.3]{Graham-Lee}.\\

We define an extension operator from  boundary metrics to interior metrics as follows.
Let $h$ be an asymptotically hyperbolic metric on $M$ with $\bar h = \rho^2h \in \CI(\bar M, \cS^2(M)),$ choose a non-negative cut-off function $\phi\in\CI(\bar M)$ supported in the set $\cU$ on which the flow along $\bar h$-geodesics normal to $\pa M$ is a local diffeomorphism, and which is identically equal to one in a neighborhood of $\pa M.$
Define
\begin{equation}\label{eq:DefT}
	E_h(\hat g) = \phi \bar g + (1-\phi)\bar h, \quad T(\hat g) = T_{\rho,h}(\hat g) = \rho^{-2}E_h(\hat g)
\end{equation}
where $\bar g$ is the extension of $\hat g$ from $\pa M$ obtained by parallel translation and requiring $\bar g(\nu, \cdot) =d\rho,$ with $\nu$ the inward pointing normal to $\pa M$ corresponding to $\bar h.$
Thus $E_h(\hat g)$ is a metric on $M$ extending $\hat g$ such that $\rho^{-2}E_h(\hat g)$ is an asymptotically hyperbolic metric on $M.$\\

Fix now $M = \bbB^{n+1},$ $\rho(\zeta) = \tfrac12(1-|\zeta|^2),$ and $h = {\df h},$ the hyperbolic metric.

Given $\hat g \in \cC^{k,\gamma}(\pa M, \cS^2(\pa M)),$ a metric on $\pa M,$ we can employ the argument in \cite[pg 203-205]{Graham-Lee} virtually unchanged to construct approximate solutions to $Q_{\alpha,\beta}(g,T(\hat g))=0.$ We start with $g_1 = \rho^{-2}E_{\df h}(\hat g),$ for which it is easy to see from \eqref{eq:LeadTerm} that $Q(g_1, \rho^{-2}E_{\df h}(\hat g))= \cO(\rho^{-1})$ and then use the linearization in Lemma \ref{lem:Linear} and the indicial root computation in \cite[\S 2]{Graham-Lee} to construct successive approximations resulting in:
\begin{proposition}[{\cite[Theorem 2.11]{Graham-Lee}}] \label{lem:GL211}
There is a smooth operator 
\begin{equation*}
	S: \cC^{k,\gamma}(\pa M, \cS^2(\pa M)) \lra \Lambda^{-2}_{k-\mu,\gamma}(M,\cS^2(M)),
\end{equation*}
where $\mu= \min(k-2,n-1),$
such that $Q(S(\hat g), T(\hat g)) = \cO(\rho^{m-1})$ and $S({\df h}) = {\df h}.$
The map $\hat g \mapsto Q(S(\hat g), T(\hat g))$ is smooth from $\cC^{k,\gamma}(\pa M, \cS^2(\pa M))$ into $\Lambda^{\mu-1}_{k-\mu-2,\gamma}(M,\cS^2(M)).$
\end{proposition}

Let $\hat {\df h}$ denote the round metric on $\bbS^{n}.$

\begin{theorem}[{\cite[Theorem 4.1]{Graham-Lee}}] \label{thm:GL41}
Let $M= \bbB^{n+1},$ $n\geq 4,$ $k\geq 2,$  $\theta\in(0,1),$ and
let $(\alpha, \beta)$ be such that
\begin{equation*}
	\sum q\lambda^{(2q)} (\alpha_q+(n+1)\beta_q) \neq 0
	\Mand
	\bA_1(\alpha, \kappa)\neq 0.
\end{equation*}
There exists $\eps>0$ such that, if $\hat g$ is a smooth metric on $\pa M$ with
\begin{equation*}
	\norm{\hat g-\hat {\df h}}_{k,\theta} <\eps,
\end{equation*}
there is a metric $g$ on $M$ with uniformly negative Ricci curvature such that 
\begin{equation*}
\begin{gathered}
	\rho^2 g \in \cC^{n-1,\gamma}(\bar M, \cS^2(M)) \Mforany 0<\gamma <\min(1-\tfrac12(n-\sqrt{n^2-8}), k+\theta-2n+2), \\
	\rho^2g\rest{T\pa M} = \hat g, 
	\quad \Mand \quad
	Q_{\alpha, \beta}(g, T(\hat g)) =0.
\end{gathered}
\end{equation*}
\end{theorem}

\begin{proof}
We summarize the proof of \cite[Theorem 4.1]{Graham-Lee}.

Set $\mu = \min(k-2,n-1),$ $\gamma_n = 1 -\tfrac12(n-\sqrt{n^2-8}),$ and let $s = \mu+1$ if $\mu+1<n$ and otherwise $s \in (n-1, n-1+\gamma_n).$
Let $L=D_1(Q_{\alpha,\beta}(g,t))_{({\df h}, {\df h})}(r),$ so that from Lemma \ref{lem:Linear}, $L$ is a non-zero multiple of $\Delta_{\df h}+2n$ on pure-trace tensors (relative to $\df h$) and a non-zero multiple of $\Delta_{\df h}-2$ on ${\df h}$-trace-free tensors.
For this choice of $s,$ \cite[Corollary 3.11]{Graham-Lee} implies that
\begin{equation*}
	L:\Lambda^{s-2}_{k-\mu,\theta}(M,\cS^2(M)) \lra \Lambda^{s-2}_{k-\mu-2,\theta}(M, \cS^2(M))
\end{equation*}
is an isomorphism and 
$Q_{(\alpha,\beta)}(S(\hat g), T(\hat g)) \in \Lambda^{s-2}_{k-\mu-2,\theta}(M, \cS^2(M)).$

Define
\begin{equation*}
\begin{gathered}
	\sB \subseteq \cC^{k,\theta}(\pa M, \cS^2(\pa M)) \times \Lambda^{s-2}_{k-\mu,\theta}(M, \cS^2(M)) \\
	\sB = \{ (\hat g, r): \hat g \text{ is pos.def. on }\pa M, \; S(\hat g) \text{ is defined, } \Mand S(\hat g)+r \text{ is pos.def. on }M\},
\end{gathered}
\end{equation*}
and a map
\begin{equation*}
\begin{gathered}
	\sQ: \sB \lra \cC^{k,\theta}(\pa M, \cS^2(\pa M)) \times \Lambda^{s-2}_{k-\mu-2,\theta}(M, \cS^2(M)) \\
	\sQ(\hat g, r) = (\hat g, Q_{(\alpha,\beta)}(S(\hat g)+r, T(\hat g))).
\end{gathered}
\end{equation*}
As in \cite[pg. 221]{Graham-Lee}, it follows from Lemma \ref{lem:GL211} and \cite[Proposition 3.3]{Graham-Lee} that $\sQ$ is smooth, satisfies $\sQ(\hat{\df h}, 0) = (\hat{\df h}, 0),$ and its linearization about $(\hat {\df h},0),$
\begin{multline*}
	D\sQ_{(\hat {\df h}, 0)}: \cC^{k,\theta}(\pa M, \cS^2(\pa M)) \times \Lambda^{s-2}_{k-\mu,\theta}(M, \cS^2(M)) \\
	\lra \cC^{k,\theta}(\pa M, \cS^2(\pa M)) \times \Lambda^{s-2}_{k-\mu-2,\theta}(M, \cS^2(M)),
\end{multline*}
is given by
\begin{multline*}
	D\sQ_{(\hat {\df h}, 0)}(\hat q, r) 
	= (\hat q, D_1Q{({\df h}, {\df h})}(DS_{\hat {\df h}}\hat q+r) + D_2Q_{({\df h}, {\df h})}(DT_{\hat {\df h}}\hat q)) 
	= (\hat q, Lr + K\hat q), \\
	\Mwhere K\hat q = D_1Q_{({\df h}, {\df h})}(DS_{\hat {\df h}}\hat q) + D_2Q_{({\df h}, {\df h})}(DT_{\hat {\df h}}\hat q)).
\end{multline*}
We have $K\hat q \in \Lambda^{s-2}_{k-\mu-2,\theta}(M, \cS^2(M))$ (since $Q_{(\alpha,\beta)}(S(\hat g), T(\hat g)) \in \Lambda^{s-2}_{k-\mu-2,\theta}(M, \cS^2(M))$ for every $\hat g$) and so the equation 
\begin{equation*}
	D\sQ_{(\hat {df h}, 0)}(\hat q, r) = (\hat w, v)
\end{equation*}
has a unique solution given by $\hat q = \hat w$ and $r = L^{-1}(v-K\hat w).$
The map $(\hat w, v) \mapsto (\hat q, r)$ is bounded as a map
\begin{equation*}
	\cC^{k,\theta}(\pa M, \cS^2(\pa M)) \times \Lambda^{s-2}_{k-\mu-2,\theta}(M, \cS^2(M)) \lra
	\cC^{k,\theta}(\pa M, \cS^2(\pa M)) \times \Lambda^{s-2}_{k-\mu,\theta}(M, \cS^2(M)),
\end{equation*}
so by the inverse function theorem $\sQ$ is locally invertible in some neighborhood of $(\hat{\df h},0).$ 
Thus if $g$ is sufficiently close to ${\df h}$ we can solve the equation $\sQ(\hat g_1, r) = (\hat g, 0),$ i.e., find $g = S(\hat g)+r$ such that
\begin{equation*}
	Q_{\alpha,\beta}(g, T(\hat g))=0.
\end{equation*}
Since $\hat g$ is smooth, $\rho^2S(\hat g)$ and $\rho^2T(\hat g)$ are smooth in $\bar M,$ so
\begin{equation*}
	g = S(\hat g)+r \in \rho^{-2}\CI(\bar M, \cS^2(M)) + \rho^{-2}\Lambda^s_{k-\mu,\theta}(M, \cS^2(M)) 
		\subseteq \cC^{2,\theta}(M, \cS^2(M)).
\end{equation*}
As $L$ is elliptic, $g \in \CI(M, \cS^2(M)).$

Since $s\geq 1,$ we always have $\rho^2 g \in \CI(\bar M, \cS^2(M))+ \Lambda^1_{2,\theta}(M, \cS^2(M)),$ so $\bar g$ is continuous on $\bar M.$ As in \cite{Graham-Lee}, using \cite[Proposition 3.3]{Graham-Lee} allows us to see that $\bar g$ is Lipschitz and in $\cC^{n-1,s-n+1}(\bar M, \cS^2(M)).$
Shrinking the neighborhood of $\hat {\df h}$ if necessary, $\bar g$ can be made arbitrarily close to $\bar h$ in the $\Lambda^0_{k-\mu,\theta}(M, \cS^2(M))$ norm and in particular, $g$ will have strictly negative Ricci curvature. 
\end{proof}

\begin{lemma}[{\cite[Lemma 2.2]{Graham-Lee}}] \label{lem:GL22}
Let $g$ be a conformally compact metric of class $\cC^3$ on $M$ such that, for some $K <0,$
\begin{equation*}
	\Ric(g)(V,V)\leq K|V|^2_g \Mforall V \in TM,
\end{equation*}
and such that in coordinates smooth up to the boundary $\pa_k \bar g_{ij}$ and $\rho\pa_k\pa_s \bar g_{ij}$ are bounded. 
Let $t$ be a conformally compact metric of class $\cC^3$ on $M,$ such that $\bar t \in \cC^2(\bar M,\cS^2(M)).$

a) If $\cR_g^{(2q)} - \lambda^{(2q)}g - \Phi_{(e_q, 0)}(g,t)=0$ then $\cR_g^{(2q)} = \lambda^{(2q)}g.$

b) If $\alpha$ is such that $\mathrm{LimSec}(\alpha)\neq \emptyset$ and $\beta_q = -\tfrac1{2q}\alpha_q,$ then
\begin{equation*}
	\sum \alpha_q (\cE^{(2q)}_g - (1-\tfrac{n+1}{2q})\lambda^{(2q)}g) -\Phi_{(\alpha, \beta)}(g,t) =0
	\implies
	\sum \alpha_q (\cE^{(2q)}_g - (1-\tfrac{n+1}{2q})\lambda^{(2q)}g) =0
\end{equation*}
\end{lemma}

\begin{proof}
In the pure Lovelock setting, since $\cR^{(2q)}_{g} + \lambda g - \Phi^{(2q)}(g,t)$ vanishes, and 
$\delta_g\cG^{(2q)}_g$ kills the first term by the second Bianchi identity and the second term by the metric property of the connection, we must have
\begin{equation*}
	\delta_g \cG^{(2q)}_g\Phi_{(e_q, 0)}(g,t) =0.
\end{equation*}
Let $\omega$ be the one-form 
$gt^{-1}\delta_g\cG^{(2)}_g (t)$ so that this equation implies the vanishing of
\begin{equation*}
\begin{multlined}
	\delta_g \cG^{(2q)}_g ((n-2q+1) \delta_g^* -(q-1) g \delta_g)\omega \\
	=
	\delta_g \Big( 
	(n-2q+1) \tfrac12(\omega_{i,j}+\omega_{j,i} - \tfrac1q g_{ij} g^{st}\omega_{s,t})
	+(q-1)(1-\tfrac m{2q})g_{ij} g^{st}\omega_{s,t} \Big) \\
	=
	\frac{(n-2q+1)}2 \delta_g  
	(\omega_{i,j}+\omega_{j,i} - g_{ij} g^{st}\omega_{s,t})
	= 
	\frac{(n-2q+1)}2 \delta_g  \cG^{(2)}_g\delta_g^*\omega.
\end{multlined}
\end{equation*}

When $\beta_q = -\tfrac1{2q}\alpha_q,$ we have
\begin{equation*}
	\Phi_{(\alpha,\beta)}(g,t)
	= - \bA_1(\alpha, \kappa)(\delta_g^* + \tfrac12g\delta_g)\omega,
\end{equation*}
and 
the second Bianchi identity implies that
\begin{equation*}
\begin{gathered}
	0 
	= \delta_g
	\Big(
	\sum \alpha_q (\cE^{(2q)}_g - (1-\tfrac{n+1}{2q})\lambda^{(2q)}g) -\Phi_{(\alpha, \beta)}(g,t) \Big) \\
	= \bA_1(\alpha, \kappa)\delta_g(\delta_g^* + \tfrac12g\delta_g)\omega
	= \bA_1(\alpha, \kappa)\delta_g\cG^{(2)}_g\delta_g^*\omega.
\end{gathered}
\end{equation*}

In either case we have
\begin{equation*}
	\delta_g\cG^{(2)}_g\delta_g^*\omega =0,
\end{equation*}
just as in the proof of \cite[Lemma 2.2]{Graham-Lee}.
As explained there, this implies that $|\omega|^2_g$ is bounded and $\Delta_g |\omega|^2_g \leq 2K|\omega|^2_g$
so \cite[Theorem 3.5]{Graham-Lee} implies $\omega=0.$
\end{proof}

As a corollary of Theorem \ref{thm:GL41} and Lemma \ref{lem:GL22} we obtain Theorem \ref{thm:GLA}.

%

\end{document}